\documentclass[12pt]{amsart}
\usepackage{graphicx}
\usepackage{amssymb}
\usepackage{amsfonts}
\usepackage{hyperref}
\usepackage{cancel}

\usepackage[normalem]{ulem}

\usepackage[margin=2cm]{geometry}
\usepackage{mathrsfs}
\swapnumbers
\usepackage[usenames,dvipsnames,svgnames,table]{xcolor}

  [1995/08/01 v0.1 Double stroke roman fonts]

\def\ds@whichfont{dsrom}
\relax

\DeclareMathAlphabet{\mathds}{U}{\ds@whichfont}{m}{n}

\numberwithin{equation}{section}
\newtheorem{theorem}{Theorem}
\newtheorem{lemma}[theorem]{Lemma}
\newtheorem{corollary}[theorem]{Corollary}

\newtheorem{proposition}[theorem]{Proposition}
\theoremstyle{definition}
\newtheorem{definition}[theorem]{Definition}
\newtheorem{assumption}[theorem]{Assumption}
\newtheorem{remark}[theorem]{Remark}
\newtheorem{example}[theorem]{Example}

\theoremstyle{plain}

\numberwithin{figure}{section} 
\theoremstyle{plain}
\theoremstyle{plain}
\theoremstyle{remark}
\newtheorem*{acknowledgement*}{Acknowledgement}
\theoremstyle{example}

\newcommand{\cB}{{\mathcal B}}
\newcommand{\cC}{{\mathcal C}}

\newcommand{\cE}{{\mathcal E}}
\newcommand{\cF}{{\mathcal F}}

\newcommand{\cL}{{\mathcal L}}

\newcommand{\cX}{{\mathcal X}}

\newcommand{\te}{{\theta}}

\newcommand{\Om}{{\Omega}}
\newcommand{\om}{{\omega}}
\newcommand{\ve}{{\varepsilon}}
\newcommand{\del}{{\delta}}
\newcommand{\Del}{{\Delta}}

\newcommand{\Gam}{{\Gamma}}

\newcommand{\sig}{{\sigma}}
\newcommand{\al}{{\alpha}}
\newcommand{\be}{{\beta}}

\newcommand{\la}{{\lambda}}

\newcommand{\bbC}{{\mathbb C}}
\newcommand{\bbE}{{\mathbb E}}

\newcommand{\bbN}{{\mathbb N}}
\newcommand{\bbP}{{\mathbb P}}
\newcommand{\bbR}{{\mathbb R}}

\def\fm{{\mathfrak{m}}}

\newcommand{\DS}{\displaystyle}

\begin{document}
\title[]{Non-uniform Berry-Esseen theorems for weakly dependent random variables}
 \author{Yeor Hafouta}
\address{The University of Florida}

\email{yeor.hafouta@mail.huji.ac.il}

\dedicatory{  }
 \date{\today}
\maketitle

\begin{abstract}
We obtain non-uniform Berry-Esseen type estimates for several classes of weakly dependent sequences of random variables, including uniformly elliptic inhomogeneous Markov chains, random and time-varying (partially) hyperbolic or expanding dynamical systems, products of random matrices and some classes of local statistics. 
As an application of the non uniform expansions we obtain average versions the Berry-Esseen theorem, which provide estimates of the underlying distribution function in $L^p(dx)$ by the standard normal distribution function.
An additional  application is to estimation  estimates of  expectations $\bbE[h(S_n)]$ of  functions $h$   of the underlying sequence $S_n$, whose derivatives grow at most polynomially fast by corresponding Gaussian expectation. In particular we provide estimates on the moments of $S_n$ by means of the corresponding moments of a standard normal random variable and the variance of  $S_n$. 
\end{abstract}

\section{Introduction}

\subsection{Berry-Esseen theorems} 
Let $S_n$ be a sequence of centered random variables so that $\sig_n=\|S_n\|_{L^2}\to\infty$. Recall that $S_n$ obeys the (self-normalized) central limit theorem (CLT) if $W_n=S_n/\sig_n$ converges in distribution to the standard normal law, namely for every $x\in\bbR$,
$
\lim_{n\to\infty}\bbP(W_n\leq x)=\Phi(x):=\frac{1}{\sqrt{2\pi}}\int_{-\infty}^x e^{-t^2/2}dt.
$ 
Recall that the optimal CLT rate in general is $O(\sig_n^{-1})$, which is the rate in the classical Berry-Esseen theorem. We thus
we say that $W_n$ obeys the  Berry-Esseen theorem  if 
$$
\sup_{x\in\bbR}\left|\bbP(W_n\leq x)-\Phi(x)\right|=O(\sig_n^{-1}).
$$ 
In this paper we will consider partial sums $S_n=\sum_{j=1}^nX_j$. In the past decades optimal CLT rates have been proven to many classes of weakly dependent  summands $X_j$.
First results for independent summands were proven in \cite{Berry, Ess} by Berry and Esseen, and  by now optimal rates have been established for several classes of weakly dependent summands
when $\sig_n^2=\text{Var}(S_n)$ grows linearly fast in $n$ (see \cite{Nag2, Stein72, CJ, RE, GH, SaulStat, RioBE, HH, GO, Jirak2016, HK} for a partial list). Three exceptions are  \cite{DolgHaf, DolgHaf PTRF 2, MarShif1} where optimal rates are obtained for additive functionals of  inhomogeneous Markov chains, for sequential dynamical systems and for inhomogeneous Markov shifts. In these papers $\sig_n^2$ can grow arbitrarily slow.


\subsection{Main results: non-uniform Berry-Esseen theorems }\label{Intro2}
As opposed to the results discussed in the previous section, this paper concerning non-uniform estimates, which for the sake of convenience are described here as definitions.
 \begin{definition}
We say that $W_n$ obeys the non-uniform Berry-Esseen theorem with power $s\geq 1$ if  there is a constant $C=C_s>0$ for that for every $x\in\bbR$ we have
$$
\left|\bbP(W_n\leq x)-\Phi(x)\right|\leq C_s(1+|x|)^{-s}\sig_n^{-1}.
$$
\end{definition}

Non uniform Berry-Esseen theorems for  partial sums of iid summands are by now a classical result, see for instance \cite{Nag65}, \cite{Osipov67},  \cite{Osipov72} and \cite{Petrov72}. See also  \cite{ChenShao2001} for a non-uniform Berry-Esseen theorems for independent but not necessarily identically distributed summands, \cite{ChenShao2004} for local dependent summands, \cite{Jirak2015, Jirak2016, Jirak2023} for  Bernoulli shifts, \cite{[DewanandRao2005]} for associated sequences and \cite{Aus} for $\alpha$-mixing stationary random fields.
We also refer to \cite{BE}, \cite{Shao} and \cite{Tu} for non-uniform bounds in different contexts. However, despite the variety of applications of non uniform Berry-Esseen theorems (described  below), the literature about non-uniform estimates for weakly dependent random variables is not as vast as the literature in the uniform case. The goal of this paper is to extend the known results from the uniform case to the non-uniform case, as will be discussed in the following sections.
 This will be done by using an abstract scheme that allows us to capture
 the following types of sequences:
\begin{itemize}

\item Partial sums $S_n=\sum_{j=0}^{n-1}g\circ T^j$ generated by a chaotic  dynamical system $T$ and a sufficiently regular observable $g$;
\vskip0.1cm

\item Partial sums $S_n=\sum_{j=1}^n g(\xi_j)$ generated by an homogeneous geometrically ergodic\footnote{Instead of geometric ergodicity, we can assume that the corresponding Markov operator has a spectral gap on an appropriate Banach space $B$, and that $g\in B$.} Markov chain $\{\xi_j\}$ and a bounded function $g$.

\vskip0.1cm
\item Partial sums generated by  random  (sequential) choatic dynamical systems and  random (sequential) sufficiently regular observables.
\vskip0.1cm
\item $S_n=\ln \|A_nA_{n-1}\cdots A_1x\|$ where $A_j$ are iid (strongly irreducible) random matrices and $x$ is a unit vector.
\vskip0.1cm
\item Partial sums $S_n=\sum_{j=1}^{n}f_j(\xi_j,\xi_{j+1})$ of additive functionals $f_j$ of uniformly elliptic (not necessarily homogeneous) Markov chains $\xi_j$.
\vskip0.1cm
\item Some classes of local statistics.
\end{itemize}
In the above examples what is currently known are only uniform Berry-Esseen theorems (i.e. when $s=0$). See, for instance \cite{DolgHaf} for inhomogeneous Markov chains, \cite{GH,GO, HH} for chaotic dynamical systems and stationary Markov chains, \cite{FP} for random matrices and  \cite[Sections 3-5]{Dor} for local statistics. The non-uniform Berry-Esseen theorems proven in this paper are new already for all of these examples.
As will be discussed in the next sections, once the appropriate non-uniform estimates are obtained several other results will follow. Let us also remark that relying on the results of a previous version of this paper non-uniform Berry-Esseen theorems were obtained for several sequential dynamical systems \cite{DolgHaf PTRF 2} and for Markov shifts \cite{MarShif1}. However, some of our examples in the random (or sequential) cases are not covered in  \cite{DolgHaf PTRF 2} and \cite{MarShif1}.

\section{Main abstract result}\label{Main}
Let $(S_n)$ be a sequence of  centered random variables so that, $\sig_n=\|S_n\|_{L^2}\to\infty$. Let us define 
$W_n=\frac{S_n-A_n}{B_n}$, where  $A_n$ is a bounded sequence and $B_n=\sig_n+O(1)$.
For instance, we can take $A_n=0$ and $B_n=\sig_n$ which corresponds to the centralized self-normalized case, but we will see in some of our application to stationary sequences and products of random matrices that the choice of $A_n=c+O(\del^n)$, for some $\del\in(0,1),\, c\in\bbR$ and $B_n=\sig\sqrt n$, $\sig>0$ is more natural.
Let $$F_n(x)=\bbP(W_n\leq x)$$ be the distribution function of $W_n$. In this section we will describe our non uniform Berry-Esseen theorems, under conditions which involve the difference between the logarithmic characteristic function of $S_n/\sig_n$ and the standard normal one, 
which is given by
$$
\Lambda_n(t)=\ln\bbE[e^{itS_n/\sig_n}]+t^2/2.
$$

We consider here the following assumption (for different values of $m$) which was introduced in \cite{DolgHaf}.
\begin{assumption}\label{GrowAssum}
For some integer $m\geq2$, for all $3\leq j\leq m+1$ there exist constants $C_j,\ve_j>0$ so that
\begin{equation}\label{GAs}
\sup_{t\in[-\ve_j\sig_n,\ve_j\sig_n]}|\Lambda_n^{(j)}(t)|\leq C_j\sig_n^{-(j-2)}.
\end{equation}
\end{assumption}
\begin{remark}
Assumption \ref{GrowAssum} only concerns the derivatives of $\ln\bbE[e^{itS_n/\sig_n}]$ (since for $j\geq 3$ the second term vanishes). However,  it is more natural to present the results using the difference $\Lambda_{n}(t)$ since it measures the deviation from normality in an appropriate sense.
\end{remark}

\begin{remark}\label{Add}
Assumption \ref{GrowAssum} clearly holds true for every $m$ when 
$$
\gamma_j(S_n/\sig_n)\leq C^j(j!)\sig_n^{-(j-2)}
$$
for some constant $C$ and all $j\geq3$,
where $\gamma_j(W)$ is the $j$-th cumulant of a random variable $W$ (defined in \eqref{CumDef}).  
This is the case for  partial sums $S_n=\sum_{j=1}^n f_j(\xi_j)$ generated by exponentially fast  $\phi$ mixing Markov chains $(\xi_j)$ and   uniformly bounded functions $f_j$ such that  $\text{Var}(S_n)$ grows linearly fast  (see \cite[Theorem 4.26]{SaulStat}). Another example are
 U-statistics of the form $S_n=\sum_{1\leq i<j\leq n}h(X_i,X_j)$ generated by iid $X_j$ with $\mathbb E[|h(X_1,X_2)|^j]\leq C^j$ (see \cite[Section 3]{Dor}). Other examples include characteristic polynomials in the circular ensembles and determinantal point processes as in \cite[Section 4-5]{Dor}. Already for these examples Theorem \ref{BE2} is new (the uniform estimates are known), and we refer the readers' to Sections  \ref{Stat} and \ref{NonStat} for a detailed description of other examples where Assumption \ref{GrowAssum} holds.
\end{remark}

Our first result is the following non-uniform Berry-Esseen theorem:
\begin{theorem}\label{BE2}
Under Assumption \ref{GrowAssum}, the non-uniform Berry-Esseen theorem with power $s=m$ holds true, namely for all $x\in\bbR$ we have
$$
\left|F_n(x)-\Phi(x)\right|\leq C\sig_n^{-1}(1+|x|)^{-m}
$$
where $C$ is some constant. In particular, for every $p>1/m$ we have
$$
\|F_n-\Phi\|_{L^p(dx)}=O(\sig_n^{-1}).
$$
\end{theorem}


\subsection{Gaussian estimates of functions with polynomially fast growing derivatives}

The following result follows from Theorem \ref{BE2}.
\begin{corollary}\label{Cor}
Let $h:\bbR\to\bbR$ be  an a.e. differentiable function so that $H_m=\int\frac{|h'(x)|}{(1+|x|)^m}dx<\infty$.  
\vskip0.1cm
Under Assumption \ref{GrowAssum}, there is a constant $R>0$ which does not depend on $h$ so that 
$$
\left|\bbE[h(W_n)]-\int h(x)\varphi(x)dx\right|\leq RH_m\sig_n^{-1}.
$$
In particular, for every $q<m$ we have 
$$
\left|\bbE[W_n^q]-\int x^q\varphi(x)dx\right|\leq RH_m\sig_n^{-1}
$$
and 
$$
\left|\bbE[|W_n|^q]-\int |x|^q \varphi(x)dx\right|\leq RH_m\sig_n^{-1}
$$ 
where $R_m$ depends only on $m$.
\end{corollary}
The proof is simple and straight forward, and so we present it here.
\begin{proof}
  We have
\begin{equation}\label{int form}
\bbE[h(W_n)]=-\bbE\left[\int_{W_n}^{\infty}h'(x)dx\right]=-\int_{-\infty}^{\infty}h'(x)\bbP(W_n\leq x)dx=
-\int_{-\infty}^{\infty}h'(x)F_n(x)dx
\end{equation}
and so if $Z$ is a standard normal random variable then
$$
\left|\bbE[h(W_n)]-\bbE[h(Z)]\right|\leq\int |h'(x)||F_n(x)-\Phi(x)|dx\leq RH_m\sig_n^{-1}.
$$   
\end{proof}



\section{Proof of Theorem \ref{BE2}}\label{SN}
\subsection{``Reduction" to the self normalized case: proof of Theorem \ref{BE2}  relying on the self-normalized case}\label{Reduc}
In this section we will prove Theorem \ref{BE2} based on the validity of the theorem in the self-normalized case when $B_n=\sig_n$ and $A_n=0$. Let $W_n=S_n/\sig_n$ and $Z_n=\frac{W_n-A_n}{B_n}$ where $B_n=\sig_n+O(1)$ and $A_n=c+O(\delta^n)$. Then 
$$
\sup_n\|W_n-Z_n\|_{L^2}<\infty.
$$
Thus, the reduction is carried out exactly like in \cite[Section 3.4]{MarShif1}.

\subsection{The Edgeworth polynomials and their Fourier transforms}\label{Form}
Let us denote by 
\begin{equation}\label{CumDef}
\gamma_j(W)=i^{-j}\left(\frac{d^j}{dt^j}\,\bbE[e^{it W}]\right)\Big|_{t=0}    
\end{equation}
 the $j$-th cumulant of a random variable $W$ with  finite absolute $j$ moments.
Let $H_s$ be the $s$-th Hermite polynomial, which is defined through the identity $(-1)^s H_s(x)\varphi(x)=\varphi^{(s)}(x)$, where $\varphi(x)=\frac{1}{\sqrt{2\pi}}e^{-\frac12x^2}$ is the standard normal density. Note that $\gamma_j(W_n)=\Lambda_n^{(j)}(0)$ and so
  by Assumption \ref{GrowAssum}, for all $3\leq j\leq m+1$,
\begin{equation}\label{Cum}
\gamma_j(W_n)=\Lambda_n^{(j)}(0)=O(\sig_n^{-(j-2)})
\end{equation}
where $\Lambda_n$ comes from the main Assumption \ref{GrowAssum}.

Next, let us consider the following polynomials 
\begin{equation}\label{P def}
P_{m,n}(z)=\sum_{\bar k}\frac1{k_1!\cdots k_{m-2}!}\left(\frac{\gamma_3(W_n)}{3!}\right)^{k_1}\cdots \left(\frac{\gamma_{m}(W_n)}{m!}\right)^{k_{m-2}}z^{3k_1+...+mk_{m-2}}
\end{equation}
where the summation runs over the collection of $m-2$ tuples of nonnegative integers $(k_1,...,k_{m-2})$ that are not all $0$ so that $\sum_{j} jk_j\leq m-2$. Let $\nu_{m,n}$ be the signed measure on $\bbR$ whose Fourier transform is 
$$
g_{m,n}(t)=e^{-t^2/2}(1+P_{m,n}(it)).
$$
Then $\nu_{m,n}$ is absolutely continuous with respect to the Lebesgue measure with density
$$
\varphi_{m,n}(x)=\varphi(x)\left(1+\sum_{\bar k}\frac1{k_1!\cdots k_{m-2}!}\left(\frac{\gamma_3(W_n)}{3!}\right)^{k_1}\cdots \left(\frac{\gamma_{m}(W_n)}{m!}\right)^{k_{m-2}}H_k(x)\right)
$$
where $\varphi(x)=\frac{1}{\sqrt{2\pi}}e^{-x^2/2}$ is the standard normal density function and $k=k(k_1,...,k_{m-2})=3k_1+...+mk_{m-2}$. Using that $H_k(x)\varphi(x)=-\left(H_{k-1}(x)\varphi(x)\right)'$ we see that the corresponding generalized distribution function is 
\begin{equation}\label{Phi m n def 1}
\Phi_{m,n}(x)=
\end{equation}
$$\int_{-\infty}^{x}\varphi_{m,n}(t)dt=\Phi(x)-\varphi(x)\sum_{\bar k}\frac1{k_1!\cdots k_{m-2}!}\left(\frac{\gamma_3(W_n)}{3!}\right)^{k_1}\cdots \left(\frac{\gamma_{m}(W_n)}{m!}\right)^{k_{m-2}}H_{k-1}(x).
$$ 
Note that $\Phi_{m,n}(-\infty)=0$ and 
$$
\Phi_{m,n}(\infty)=\int_{-\infty}^{\infty}\varphi_{n,m}(x)dx=g_{m,n}(0)=1.
$$
Notice also that for $m=2$ we have $P_{2,n}=0$ and so $\nu_{2,n}$ is the standard normal law and $\Phi_{2,n}=\Phi$ is the standard normal distribution function.

Next, observe that the function $\Phi_{n,m}(x)$ can also be written in the form
\begin{equation}\label{BoldP}
\Phi_{n,m}(x)=\Phi(x)-\varphi(x)\sum_{r=1}^{m-2}\sig_n^{-r}\textbf{H}_{r,n}(x)
\end{equation}
where
\begin{equation}\label{Poly def 1.1}
\textbf{H}_{r,n}(x)=\sum_{\bar k\in A_{r}}C_{\bar k}\prod_{j=1}^{s}\left(\gamma_{j+2}(S_n)\sig_n^{-2}\right)^{k_j}H_{k-1}(x)
\end{equation}
and $A_{r}$ is the set of all  tuples of nonnegative integers $\bar k=(k_1,...,k_{s}), k_s\not=0$ for some $s=s(\bar k)\geq 1$ so that $\sum_{j}jk_j=r$ (note that when $r\leq m-2$ then $s\leq m-2$ since $k_s\geq1$). Moreover, $k=3k_1+...(s+2)k_s$ and 
 $$
 C_{\bar k}=\prod_{j=1}^{s}\frac{1}{k_j!(j+2)^{k_j}}.
 $$
 We note that  the polynomials
$\textbf{H}_{r,n}$ have bounded coefficients (because of \eqref{Cum}) and that  their degrees does not depend on $n$. The  polynomails $\textbf{H}_{r,n}$ coincide with the ``Edgeworth polynomials" defined in \cite{DolgHaf} (denoted there by $P_{r,n}$).

We also note that the  Fourier transform of the derivative of $\Phi_{m,n}(x)$ has the form 
$$
g_{m,n}(t)=e^{-t^2/2}\left(1+\sum_{r=1}^{m-2}\sig_n^{-r}\textbf{P}_{j,n}(t)\right)
$$
where 
$$
\textbf{P}_{r,n}(t)=\sum_{\bar k\in A_{r}}C_{\bar k}\prod_{j=1}^{s}\left(\gamma_{j+2}(S_n)\sig_n^{-2}\right)^{k_j}(it)^{3k_1+...+(s+2)k_{s}}.
$$

\subsection{Proof of Theorem \ref{BE2} in the self normalized case}

The proof of Theorem \ref{BE2}   relies on the following two results, whose proof is postponed to the next sections.


\begin{proposition}\label{Prp17.1}
Under Assumption \ref{GrowAssum}, 
if $|t|\geq C\sig_n^{1/3}$ for some $C>0$ then for  $p=0,1,2,...,m$ we have
$$
\left|g_{m,n}^{(p)}(t)\right|\leq C_m\sig_n^{-(m+1)}e^{-c_0t^2}
$$
where $C_m$ depends only on $C,m$ and the constants in Assumption \ref{GrowAssum} and $c_0\in(0,1/2)$ is a constant (which, upon increasing $C_m$, can be made arbitrarily close to $1/2$).
\end{proposition}

\begin{proposition}\label{Prp3}
Under Assumption \ref{GrowAssum}, 
there are constants $c,c_0>0$ and $C_m>0$ so that if $|t|\leq c\sig_n$ then for $p=0,1,2,...,m$ we have 
$$
\left|\frac{d^p}{dt^p}\left(f_n(t)-g_{m,n}(t)\right)\right|\leq C_m e^{-c_0t^2}\min\{1, |t|^{m+1-p}\}\sig_n^{-(m-1)}.
$$
\end{proposition}


\subsection{Proof of Theorems \ref{BE2}   relying on Propositions \ref{Prp17.1} and \ref{Prp3}}

\begin{lemma}\label{NonULemma}
Under Assumption \ref{GrowAssum},
for all $x\in\bbR$ and $3\leq p\leq m$ we have
$$
\left|\Phi_{p,n}(x)-\Phi_{p-1,n}(x)\right|\leq C_m(1+|x|)^{-m-1}\sig_n^{-(p-2)}.
$$
As a consequence, for all $2\leq m_1<m$,
$$
\left|\Phi_{m,n}(x)-\Phi_{m_1,n}(x)\right|\leq A_m(1+|x|)^{-m-1}\sig_n^{-(m_1-1)}.
$$
Here $A_m,C_m$ are constants that do not depend on $n$ or $x$.

\end{lemma}

\begin{proof}
Notice that the difference between the derivatives of $\Phi_{p,n}(y)$ and $\Phi_{p-1,n}(y)$ has the form
$$
\varphi_{p,n}(y)-\varphi_{p-1,n}(y)=\frac{e^{-y^2/2}}{\sqrt{2\pi}}\left(\sum_{\bar k}\frac{1}{\prod_{j=1}^{p-2}k_j!(j!)^{k_j}}\prod_{j=1}^{p-2}\left(\gamma_{j+2}(W_n)\right)^{k_j}H_k(y)\right)
$$ 
where the summation is running over all tuples $\bar k=(k_1,...,k_{p-2})$
 of nonnegative integers so that $\sum_{j}j k_j=p-2$. 
 Using \eqref{GAs} with $t=0$, we see that the product of the cumulants in the above difference is of order $\sig_n^{-(p-2)}$. Therefore,
$$
\left|\Phi_{p,n}(x)-\Phi_{p-1,n}(x)\right|\leq \int_{-\infty}^{x}|\varphi_{p,n}(y)-\varphi_{p-1,n}(y)|dy\leq 
C\sig_n^{-(p-2)}\int_{-\infty}^{x}e^{-y^2/2}(1+|y|^s)dy 
$$
where $s=s_m$ is some positive integer. This proves that the lemma when $x\leq 0$ since the integral on the right hand side decays exponentially fast in $|x|$ as $x\to-\infty$.
To prove the statement for positive $x$'s we recall that $\Phi_{j,n}(\infty)=1$,
and so 
$$
\Phi_{p,n}(x)-\Phi_{p-1,n}(x)=(1-\Phi_{p-1,n}(x))-(1-\Phi_{p,n}(x))$$
$$
=(\Phi_{p-1,n}(\infty)-\Phi_{p-1,n}(x))-(\Phi_{p,n}(\infty)-\Phi_{p,n}(x))=
\int_{x}^{\infty}\left(\varphi_{p-1,n}(y)-\varphi_{p,n}(y)\right)dy.
$$
Using this inequality  the proof proceeds similarly to the case when $x\to-\infty$.
\end{proof}

\begin{proof}[Proof of Theorem \ref{BE2}]
By the results described in \cite[Ch. VI, Lemma 8]{Petrov72}, if $F$ is a distribution function and $G$ is a  generalized distribution function so that $G(-\infty)=F(-\infty)=0$,  $G(\infty)=F(\infty)=1$, $|G'(x)|\leq K(1+|x|)^{-s}$ for some constant $K$ and all $x\in\bbR$ and
$$
\int_{-\infty}^{\infty}|x|^s|d(F(x)-G(x))|<\infty
$$
then for all $x\in\bbR$ and $T>1$,
\begin{equation}\label{Ess}
|F(x)-G(x)|\leq c(s)(1+|x|)^{-s}\left(\int_{-T}^{T}\left|\frac{f(t)-g(t)}{t}\right|dt+\int_{-T}^{T}\left|\frac{a_s(t)}{t}\right|dt+\frac{K}{T}\right)
\end{equation}
where $f(t)=\int_{-\infty}^{\infty} e^{itx}dF(x)$, $g(t)=\int_{-\infty}^{\infty} e^{itx}dG(x)$ and 
$$
a_s(t)=\int_{-\infty}^{\infty}e^{itx}d\left(x^s(F(x)-G(x))\right).
$$
Here $c(s)$ is an absolute constant which depends only on $s$.
Notice that for $t\not=0$ we have (see \cite{Bobkov2016}),
\begin{equation}\label{aaok}
a_s(t)=(i)^{-s}\frac{d^s}{dt^s}\,\frac{a_0(t)}{t}=i^{-s}\int_{0}^1\left(a_0^{(s)}(t)-a_0^{(s)}(\eta t)\right)s\eta^{s-1}d\eta.
\end{equation}
Note also that $a_0(y)=f_n(y)-g_{m,n}(y)$. Moreover, observe that for every measurable nonnegative function $q:\bbR\to\bbR$ and all $T>0$ and $s\geq1$ we have 
$$
\int_{0}^{T}\frac 1t\int_{0}^1 q(\eta t)s\eta^{s-1}d\eta dt=\int_{0}^{T}q(x)\left(\frac 1x-\frac 1T\cdot (x/T)^{s-1}\right)dx\leq\int_{0}^{T}\frac{q(x)}{x}dx.
$$ 
Indeed, the above follows by making the change of variables $t\eta=x$ in the inner integral and then changing the order of  integration.
Therefore, using also \eqref{aaok} we see that
\begin{equation}\label{aaaok}
\int_{-T}^{T}\left|\frac{a_s(t)}{t}\right|dt\leq 2\int_{-T}^{T}\left|\frac{a_0^{(s)}(t)}{t}\right|dt.
\end{equation}

Now, in order to prove the non-uniform Berry-Esseen theorem (Theorem  \ref{BE2})   
we take $F$ to be the distribution function of $W_n$,  $G=\Phi_{m,n}$, 
$T=O(\sig_n)$ and $s=m$. Note that the derivative of $G$ is uniformly bounded in $n$ because of Assumption \ref{GrowAssum}, which guarantees that $G'(x)=E_n(x)e^{-x^2/2}$ for some polynomial $E_n(x)$ whose degree does not depend on $n$ and its coefficients are uniformly bounded in $n$.
Next, by taking $p=0$ in Proposition \ref{Prp3} we see that the first integral in the brackets on the right hand side of \eqref{Ess} is of order $O(\sig_n^{-(m-1)})$. In order to estimate the second  integral in the brackets on the right hand side of \eqref{Ess} let us  take $p=m$ in Proposition \ref{Prp3}, which together with \eqref{aaaok} yields that 
$$
\int_{-T}^{T}\left|\frac{a_s(t)}{t}\right|dt=O(\sig_n^{-(m-1)}).
$$
To complete the proof of Theorem \ref{BE2} we need to pass from $\Phi_{m,n}$ to $\Phi=\Phi_{2,n}$, and this is done
by applying Lemma \ref{NonULemma} with $m_1=2$. 
\end{proof}

\subsection{Proof of the main Propositions \ref{Prp17.1} and \ref{Prp3}}

The general structure of both proofs share some similarities with existing proofs for partial sums of independent random variables (see the recent survey \cite{BobkovInd}),  but many  arguments are different since the desired error term is $o(\sig_n^{-r}), r=m-2$ and  not the Lyapunov coefficient of order $m$, and instead of independence of $S_{j+1}-S_j$ we only use Assumption \ref{GrowAssum}.
\begin{proof}[Proof of Proposition \ref{Prp17.1}]
Fix some $t$ so that $|t|\geq C\sig_n^{1/3}$. Now, since $g_m=g_{m,n}$ can be naturally extended to a complex function $g_m(z)$ on the complex plane, using the Cauchy integral formula (or the maximum modulus principle) we see that for all $\eta\in(0,1/4)$  and all $p=0,1,2,,...,m$ we have   
\begin{equation}\label{1}
|g_m^{(p)}(t)|\leq C_m\max_{|z-t|\leq \eta|t|}|z|^{-p}|g_m(z)|
\end{equation}
where have used that $|t|/2\leq |z|\leq 2|t|$ which follows from the inequality $|z-t|\leq \eta|t|$.
Thus, it is enough to show that when $|t|\geq C\sig_n^{1/3}$, 
$$
\max_{|z-t|\leq \eta|t|}|z|^{-p}|g_m(z)|\leq C_m\sig_n^{-(m+1)}e^{-c_0t^2}
$$
where $c_0\to 1/2$ as $\eta\to 0$ (and it does not depend on $t$) and $C_m$ is a constant. 

Next, let us consider the entire functions $R_k(z)=e^{-z^2/2}z^k$. Then, using that 
$$
\left|\gamma_j(W_n)\right|\leq C_j\sig_n^{-(j-2)},\, 3\leq j\leq m
$$
we see that 
\begin{equation}\label{2.1}
|g_m(z)|\leq C_m\left(R_0(z)+\sum_{\bar k}\sig_n^{-d_{\bar k}}|R_{s_{\bar k}}(z)|\right)
\end{equation}
where $\bar k=(k_1,k_2,...,k_{m-2})$ ranges over all the tuples so that $d_{\bar k}:=\sum_j jk_j\leq m-2$
and $s_{\bar k}=\sum_{j}(j+2)k_j$. 
Notice that $\Re(z^2)\leq a\Re(t^2)$ for some constant\footnote{We can take $a=a(\eta)=1-2\eta-\eta^2$} $a>0$ which converges to $1$ as $\eta\to 0$ (and it does not depend on $t$). Thus, 
$$
|R_{s_{\bar k}}(z)|\leq 2e^{-at^2/2}|t|^{s_{\bar k}}.
$$
Now, since $|t|\geq C\sig_n^{1/3}$, we have that $|t|$ it is bounded away from the origin, and we conclude that there is an integer $u_m$ which depends only on $m$, $C$ and the first index $n_0$ so that $\sig_{n_0}\geq 1$ with the property that
 for all relevant tuples $\bar k$ and for all complex $z$ with $|z-t|\leq \eta|t|$ we have
$$
|R_{s_{\bar k}}(z)|\leq B|t|^{p+u_m}e^{-at^2/2}\leq B(m,a,a')|t|^{p-3(m+1)}e^{-a't^2/2}
$$
where $a'<a$ is arbitrarily close to $a$, $B$ is some constant and $B_m(m,a,a')$ depends only on $m,a,a'$ and $C$ and $n_0$ as above. Finally, since $|t|\geq C\sig_n^{1/3}$ we have 
$$
|t|^{-3(m+1)}\leq C^{-3(m+1)}\sig_n^{-(m+1)}
$$
and hence there is a constant $A$ so that
$$
|z|^{-p}|R_{s_{\bar k}}(z)|\leq A\sig_n^{-(m+1)}e^{-a't^2/2}.
$$
Now the proposition follows from \eqref{1} and \eqref{2.1}.
\end{proof}

\subsection{Proof of Proposition \ref{Prp3}}
The proof is somehow lengthy, and we will split it into several parts.

\subsubsection{Bounds on the characteristic function and its derivatives}
\begin{lemma}\label{Prp16.1}
(i) Under Assumption \ref{GrowAssum} we have the following. For every $c\in(0,1)$ there exists a constant $a>0$ so that on the domain $\{|t|\leq a\sig_n\}$ we have 
$$
|f_n(t)|\leq e^{-ct^2}.
$$

(ii)  Under Assumption \ref{GrowAssum} (with some $m$) ewe have the following. For every $c\in(0,1)$ there exist constants $a,b>0$ so that on the domain  $\{|t|\leq a\sig_n\}$ for all
  $p\leq m$ we have
 $$
 \left|\frac{d^p}{dt^p}f_n(t)\right|\leq b(1+|t|^p)e^{-ct^2}.
 $$
\end{lemma}

\begin{proof}
We have 
$$
f_n(t)=e^{\Lambda_n(t)-t^2/2}.
$$
To prove (i) we notice that $\Lambda_n(0)=\Lambda_n'(0)=\Lambda''(0)=0$. Hence, by taking $j=3$ in \eqref{GAs} we see that the second order Taylor reminder of the function $\Lambda_n$ around the origin is of order $O(|t|^3\sig_n^{-1})$. Thus, if $|t|/\sig_n$ is small enough we have 
$$
|\Lambda_n(t)|\leq C|t|^3\sig_n^{-1}\leq\ve t^2
$$
where $\ve$ is an arbitrary small positive number.

In order to prove the second part, recall that by  Fa\`a Di Bruno's formula,  for a $p$-times differential function $h(t)$ we have 
\begin{equation}\label{FDB}
\frac{d^p\,\big(e^{h(t)}\big)}{dt^p}=e^{h(t)}\sum_{\bar k}C_{p,\bar k}\prod_{j=1}^p\left(h^{(j)}(t)\right)^{k_j}
\end{equation}
where $\bar k=(k_1,...,k_p)$ ranges over all $p$-tuples of nonnegative integers so that 
$\sum_j j k_j=p$
and 
$$
C_{\bar k,p}=\frac{p!}{\prod_{j=1}^p k_j!(j!)^{k_j}}.
$$
Here $h^{(j)}$ denotes the $j$-th derivative of a function $h$.
Let us apply the above formula with the function $h(t)=\Lambda_n(t)-t^2/2$. We first claim that all the derivatives of orders $1\leq j\leq p$ of $h$ are bounded by some constant $C$ which does not depend on $n$, when $|t|\leq a\sig_n$ for some $a$ small enough. 
Plugging in these upper bounds in  \eqref{FDB} (i.e. replacing the derivatives with the upper bounds) and using part (i) to bound $|f_n(t)|=|e^{h(t)}|$  we obtain the estimates described in part (ii).

Now, let us shows that there is an $a>0$ so that on $\{|t|\leq a\sig_n\}$ we have $\max_{1\leq j\leq p}|h^{(j)}(t)|=O(1)$.
First, by Assumption \ref{GrowAssum} on a domain of the form $\{|t|\leq a\sig_n\}$ for
all $3\leq j\leq m$ we have 
$$
|h^{(j)}(t)|=|\Lambda_n^{(j)}(t)|\leq C\sig_n^{-(j-2)}=O(1).
$$
For $j=2$ we have 
$$
|h''(t)|=|1+\Lambda_n''(t)|.
$$
Since $\Lambda_n''(0)=0$ we have 
$$
\Lambda_n''(t)=\int_{0}^t\Lambda_n'''(s)ds
$$
and so by Assumption \ref{GrowAssum},
$$
|\Lambda_n''(t)|\leq C|t|\sig_n^{-1}=O(1).
$$
Combining this with the fact that $\Lambda_n'(0)=0$
we get that when $|t|=O(\sig_n)$ then
$$
|h'(t)|=|\Lambda_n'(t)-t|\leq |t|+\left|\int_{0}^{t}\Lambda_n''(s)ds\right|\leq |t|+C|t|^2\sig_n^{-1}\leq C'|t|. 
$$

\end{proof}


\subsubsection{Expansions of the characteristic function}
We first need is the following simple observation, which for the sake of convenience is formulated as a lemma.
\begin{lemma}\label{Taylor Lemma}
Let $h:(-a,a)\to\bbR, a>0$ be a function which is differentiable $k$ times for some $k$. Let $R_k$ be the Taylor remainder of $h$ of order $k$ around the origin, and let $u\leq k$. Then the $u$-th derivative $R_k^{(u)}$ of $R_k$ is the Taylor remainder of the function $h^{(u)}$ of order $k-u$ around the origin. 
\end{lemma}
\begin{proposition}\label{Prp8.2}
Under Assumption \ref{GrowAssum} we have the following. There is a constant $c_1>0$ so that for all real $t$ with $|t|\leq c_1\sig_n^{\frac{m+1}{m-1}}$ we have
$$
e^{t^2/2}f_n(t)=e^{Q_{m,n}(it)}\left(1+\ve_{n,m}(t)\right)
$$
where 
\begin{equation}\label{Q def}
Q_{m,n}(z)=\sum_{\ell=3}^m\gamma_\ell(W_n)\frac{z^\ell}{\ell!}=\sum_{\ell=3}^m\sig_n^{-\ell}\gamma_\ell(S_n)\frac{z^\ell}{\ell!}
\end{equation}
and for all $p\leq m$, 
$$
\left|\frac{d^p(\ve_{n,m}(t))}{dt^p}\right|\leq C_m|t|^{m+1-p}\sig_n^{-(m-1)}
$$
where $C_m$ depends only on $m$ and the constants in Assumption \ref{GrowAssum}.
\end{proposition}
\begin{proof}
Set 
$$
r(t)=\Lambda_n(t)-Q_{m,n}(it)
$$
which is the Taylor remainder of order $m$  of $\Lambda_n(t)$ around the origin.
Then
$$
e^{t^2/2}f_n(t)=e^{\Lambda_n(t)}=e^{Q_{m,n}(t)}e^{r(t)}.
$$
Now, by Lemma \ref{Taylor Lemma}, for all $j\leq m$ we have that $r^{(j)}(t)$ is the Taylor  remainder of order $m-j$ of the function $\Lambda_n^{(j)}(t)$ around the origin  . Using now the Lagrange form of the remainder (of order $m-j$) together with Assumption \ref{GrowAssum} we see that  when $|t|\leq c\sig_n$ for $c$ small enough then 
\begin{equation}\label{r bound}
|r^{(p)}(t)|\leq C|t|^{m+1-j}\sig_n^{-((m+1)-2)}=C|t|^{m+1-j}\sig_n^{-(m-1)}
\end{equation}
for some constant $C$.

It remains to show that when $|t|=O(\sig_n^{\frac{m+1}{m-1}})$ we can replace $e^{r(t)}$ with $1+r(t)$. Let $\ve(t)=e^{r(t)}-1$. Then  
$$
e^{t^2}f_n(t)=e^{Q_{m,n}(t)}e^{r(t)}=e^{Q_{m,n}(t)}(1+\ve(t))
$$
and we need to show that  the function $\ve(t)=\ve_{m,n}(t)=\ve(t)$ satisfies the properties described in Proposition \ref{Prp8.2}.
By taking $p=0$ in \eqref{r bound} we see that $|r(t)|\leq C|t|^{m+1}\sig_n^{-(m-1)}=O(1)$
and so 
$$
|\ve(t)|=|e^{r(t)}-1|\leq C'|r(t)|\leq C''|t|^{m+1}\sig_n^{-(m-1)}.
$$
To get the desired estimates for higher order derivatives, by \eqref{FDB}, for  $p\geq1$ we have
$$
\ve^{(p)}(t)=e^{r(t)}\sum_{\bar k}C_{p,\bar k}\prod_{j=1}^p\left(r^{(j)}(t)\right)^{k_j}=
(1+\ve(t))\sum_{\bar k}C_{p,\bar k}\prod_{j=1}^p\left(r^{(j)}(t)\right)^{k_j}
$$ 
and so, using also \eqref{r bound}  we see that there is a constant $C_m>0$ so that 
$$
\left|\ve^{(p)}(t)\right|\leq C_m|t|^{-p}\sum_{\bar m}\left(|t|^{m+1}\sig_n^{-(m-1)}\right)^{\sum_j k_j}
$$
where we have taken into account that $\sum_ j jk_j=p$. Finally, since $\sum_j k_j\geq 1$ and  $|t|^{m+1}\sig_n^{-(m-1)}\leq c_1$ for some constant $c_1$ we see that
$$
\left|\ve^{(p)}(t)\right|\leq A_m|t|^{m+1-p}\sig_n^{-(m-1)}
$$
where $A_m$ is another constant which depends on $m$ and $c_1$, and the proof of the proposition is complete.
\end{proof}

\subsubsection{Estimates on the cumulants polynomials $Q_m$}
\begin{lemma}\label{Lem11.1}
Under Assumption \ref{GrowAssum} we have 
$$
|Q_{m,n}(z)|\leq C_m|z|^3\sig_n^{-1}\sum_{\ell=3}^m|z|^{\ell-3}\sig_n^{-(\ell-3)}.
$$
In particular, if $|z|\leq C\sig_n$ then 
$$
|Q_{m,n}(z)|\leq C'|z|^3\sig_n^{-1}
$$
for some constant $C'$, 
and so when $|z|\leq C_1\sig_n^{1/3}$ then $|Q_{m,n}(z)|\leq C_1'$. 
\end{lemma}
\begin{proof}
The lemma follows from the definition of $Q_{m,n}$ together with the estimates 
$$
|\gamma_j(W_n)|\leq C_j\sig_n^{-(j-2)},\, 3\leq j\leq m+1.
$$
on the cumulants of $W_n=S_n/\sig_n$ (which come from Assumption \ref{GrowAssum}).
\end{proof}

The next result we need is as follows.
\begin{lemma}\label{Lem11.2}
Under Assumption \ref{GrowAssum}, if $|t|\leq C\sig_n^{1/3}$ for some constant $C$ then 
$$
e^{Q_{m,n}(it)}=\sum_{k=0}^{m-2}\frac{Q_m(it)^k}{k!}+\ve(t)
$$
where for all $p\leq m$ we have
$$
\left|\ve^{(p)}(t)\right|\leq A|t|^{3(m-1)-p}\sig_n^{-(m-1)}
$$
for some constant $A$ that may depend on $C$ and $m$.
\end{lemma}
\begin{proof}
By Lemma \ref{Lem11.1}, when $|z|=O(\sig_n^{1/3})$ then $Q_{m,n}(z)$ is uniformly bounded in a ball of radius $M$ which is independent of $n$. Let $\Psi(z)=e^z-\sum_{j=0}^{m-2}\frac{z^j}{j!}$. Then by the Lagrange form of the Taylor remainder of order $m-2$ of the function $e^{z}$ we have
$$
\sup_{|z|\leq M}|\Psi(z)|\leq C_M|z|^{m-1}.
$$ 
Thus, by using the penultimate in Lemma  \ref{Lem11.1} we see that if $|z|=O(\sig_n^{1/3})$ then
$$
|\ve(z)|=|\Psi(Q_{m,n}(iz))|\leq C_M|Q_{m,n}(iz)|^{m-1}\leq C'|z|^{3(m-1)}\sig_n^{-(m-1)}.
$$
This proves the lemma for $p=0$. The proof when $p>0$ is completed by using that by the Cauchy integral formula we have
$$
|\ve^{(p)}(t)|\leq p!(|t|/4)^{-p}\max_{|z-t|=|t|/4}|\ve(z)|.
$$
\end{proof}

The following result is an immediate corollary of Lemma \ref{Lem11.1} and Lemma \ref{Lem11.2},
\begin{corollary}\label{Lemma 11.3}
Under Assumption \ref{GrowAssum}, on  domains of the form
 $\{|t|\leq C\sig_n^{1/3}\}$ all the first $m$ derivatives of $e^{Q_m(it)}$ are bounded by a constant which does not depend on $n$ and $t$.
\end{corollary}

\subsubsection{Passing from $Q$ to $P$}
\begin{proposition}\label{Prp12.1}
Let $P_{m,n}$ be as in \eqref{P def}.
Under Assumption \ref{GrowAssum}, 
if $|t|\leq C\sig_n^{1/3}$ then
$$
e^{Q_{m,n}(it)}=1+P_{m,n}(it)+\del(t)
$$
where for all $p\leq m$ we have
$$
\left|\del^{(p)}(t)\right|\leq A\max(|t|^{m+1-p},|t|^{3(m-1)-p})\sig_n^{-(m-1)}
$$
for some constant $A$. 
\end{proposition}
\begin{proof}
Using Lemma \ref{Lem11.2}, we only need to bound the first $p$ derivatives of 
$$
d(t):=\sum_{k=0}^{m-2}\frac{Q_{m,n}(it)^k}{k!}-P_{m,n}(it).
$$
Notice that by the definitions \eqref{P def} and \eqref{Q def} of $P_{m,n}$ and $Q_{m,n}$, respectively, we have
$$
d(t)=\sum_{\bar k}D_{\bar k,m-2}\prod_{j=1}^{m-2}\left(\big(\gamma_{j+2}(W_n)\big)^{k_j}\right)(it)^{3k_1+4k_2+...+mk_{m-2}}
$$
where 
$$
D_{\bar k,m-2}=\frac{1}{\prod_{j=1}^{m-2} k_j!((j+2)!)^{k_j}}
$$
and $\bar k=(k_1,...,k_{m-2})$ ranges over all the tuples of nonnegative integers so that $\sum_{j}k_j\leq m-2$ and $k_1+2k_2+...+(m-2)k_{m-2}\geq m-1$. It is enough to show that the $p$-th derivatives of each monomial 
$$
w_{\bar k}(t)=\prod_{j=1}^{m-2}\left(\big(\gamma_{j+2}(W_n)\big)^{k_j}\right)(it)^{3k_1+4k_2+...+mk_{m-2}}
$$
admit the desired upper bounds. Let us fix such a monomial. Then 
$$
\left|\frac{d^p}{dt^p}\,w_{\bar k}(t)\right|\leq \prod_{j=1}^{m-2}\left|\big(\gamma_{j+2}(W_n)\big)^{k_j}\right||t|^{3k_1+4k_2+...+mk_{m-2}-p}.
$$
Notice that 
\begin{equation}\label{U}
3k_1+4k_2+...+mk_{m-2}=2\sum_{j}k_j+\sum_{j}jk_j\geq m+1
\end{equation}
and so the above power of $|t|$ is positive.

Next, let us suppose that $|t|\leq1$. Since 
$$
3k_1+4k_2+...+mk_{m-2}-p\geq m+1-p
$$
and by Assumption \ref{GrowAssum}
\begin{equation}\label{U0}
|\gamma_j(W_n)|\leq C_j\sig_n^{-(j-2)},\, 3\leq j\leq m+1
\end{equation}
and $k_1+2k_2+...+(m-2)k_{m-2}\geq m-1$
we conclude that 
$$
\left|\frac{d^p}{dt^p}\,w_{\bar k}(t)\right|\leq C|t|^{m+1-p}\sig_n^{-(m-1)}.
$$

To complete the proof, let us consider the case when  $1\leq |t|\leq c\sig_n^{1/3}$ for some $c>0$.
Then, using that $\sum_j jk_j\geq m-1>\sum_j k_j$ together with \eqref{U0} we have 
$$
\left|\frac{d^p}{dt^p}\,w_{\bar k}(t)\right|\leq C|t|^{-p}|t|^{\sum_j  jk_j+2\sum_j k_j}\sig_n^{-\sum_j jk_j}
$$
$$
=C|t|^{3(m-1)-p}\left(|t|^{\sum_j  jk_j+2\sum_j k_j-3(m-1)}\sig_n^{-(\sum_j jk_j-(m-1))}\right)\sig_n^{-(m-1)}
$$
$$
\leq C'|t|^{3(m-1)-p}\left(|t|^{\sum_j  jk_j+2\sum_j k_j-3(m-1)}|t|^{-3(\sum_j jk_j-(m-1))}\right)\sig_n^{-(m-1)}
$$
$$
=C'|t|^{3(m-1)-p}\left(|t|^{2\sum_j k_j-2\sum_j j k_j}\right)\sig_n^{-(m-1)}\leq C'|t|^{3(m-1)-p}\sig_n^{-(m-1)}.
$$
 \end{proof}

\begin{corollary}\label{Cor12.2}
Under Assumption \ref{GrowAssum},
if $|t|\leq C\sig_n^{1/3}$ then there is a constant $A>0$ so that for all $p\leq m$,
$$
\left|P_{m,n}^{(p)}(it)\right|\leq A\min(1,|t|^{-p})
$$
\end{corollary}
\begin{proof}
Notice that $P_{m,n}(it)=e^{Q_{m,n}(it)}-1-\del(t)$. Now the corollary follows from the combination of Lemma \ref{Lem11.1}, Lemma \ref{Lem11.2} and Proposition \ref{Prp12.1}. The upper bound $A$ when $|t|\leq 1$ follows directly from the definition  \eqref{P def} of $P_{m,n}$ together with Assumption \ref{GrowAssum} which yields that the coefficients of $P_{m,n}$ are bounded. 
\end{proof}

\subsubsection{Completing the proof of Proposition \ref{Prp3}}

\begin{lemma}\label{Lem13.1}
Under Assumption \ref{GrowAssum},
if $|t|\leq C\sig_n^{1/3}$ then 
$$
f_n(t)=e^{-t^2/2}\left(1+P_m(it)+r(t)\right)
$$
with 
$$
\left|r^{(p)}(t)\right|\leq C_m\sig_n^{-(m-1)}\max(|t|^{m+1-p},|t|^{3(m-1)-p}),\,p=0,1,...,m.
$$
\end{lemma}
\begin{proof}
Using Propositions \ref{Prp8.2} and \ref{Prp12.1} we have 
$$
f_n(t)=e^{-t^2/2}e^{Q_m(it)}\left(1+\ve_{m,n}(t)\right)=e^{-t^2/2}\left(1+P_m(it)+\del(t)\right)\left(1+\ve_{m,n}(t)\right).
$$
Thus 
$$
r(t)=\left(1+P_m(it)\right)\ve_{m,n}(t)+\del(t)\left(1+\ve_{m,n}(t)\right).
$$
Now the bounds on the derivative of $r(t)$ follow from Propositions \ref{Prp8.2} and \ref{Prp12.1}  and Corollary \ref{Cor12.2} together with the binomial formula for the derivatives of products of two functions.
\end{proof}

\begin{corollary}\label{Prp13.2}
If $|t|\leq C\sig_n^{1/3}$ then 
$$
\left|\frac{d^p}{dt^p}\left(f_n(t)-g_{m,n}(t)\right)\right|\leq C_m\sig_n^{-(m-1)}\max(|t|^{m+1-p},|t|^{3(m-1)+p})e^{-c_0t^2}
$$ 
for some constants $C_m,c_0>0$.
\end{corollary}
\begin{proof}
In the notations of Lemma \ref{Lem13.1} we have 
$$
f_n(t)-g_{m,n}(t)=e^{-t^2/2}r(t).
$$
Now we can use the binomial formula for the derivatives of a product of two functions and Lemma \ref{Lem13.1}.
\end{proof}

\begin{proof}[Completion of the proof of Proposition \ref{Prp3}]
 On the interval $|t|\leq C\sig_n^{1/3}$ for some $C>0$ the desired estimate follows from Corollary \ref{Prp13.2}. 
On the intervals $C\sig_n^{1/3}\leq |t|\leq C_1\sig_n$ (for some $C_1$ small enough), the desired estimate follows from Lemma \ref{Prp16.1} together with Proposition \ref{Prp17.1}.
\end{proof}

\section{Examples with  some stationary/homogeneity}\label{Stat}
In this section we will provide examples when Assumption \ref{GrowAssum} holds with all $m$ for several classes of stationary processes. 
We stress that our results are new for all of these procuresses, and what is currently known in literature is only the uniform bounds 
$$
\sup_{x\in\mathbb R}\left|F_n(x)-\Phi(x)\right|\leq C\sig_n^{-1}
$$
where $F_n$ is the distribution function of the underlying sequence. 
The fact that Assumption \ref{GrowAssum} holds in these examples essentially follows from a general functional analytic scheme that is known to apply to these examples (see Sections \ref{Sch} and \ref{Sec5.2} below). Thus, the main novelty is that we get the non-uniform estimates as in Theorem \ref{BE2} and not that Assumption \ref{GrowAssum} holds.

\subsection{General scheme}\label{Sch}
In this section we describe a functional analytic framework that in the next section will be ``verified" for several classes of stationary sequences. 

\begin{assumption}\label{GenAss}
There are analytic functions $\Pi(z),U(z)$ and $\del_n(z)$ of a complex variable $z$ which are defined on a complex neighborhood of $0$ so that:  
\begin{enumerate}
\item $\Pi(0)=0$, $\Pi''(0)>0$ and $U(0)=1$;
\vskip0.2cm
\item We have
$$
|\del_n(z)|\leq C\del^n
$$
for some $\del\in(0,1)$ and $C>0$;
\vskip0.2cm
\item We have
\begin{equation}\label{Stat assmp}
\bbE[e^{zS_n}]=e^{n\Pi(z)}(U(z)+\del_n(z))
\end{equation}
\end{enumerate}
\end{assumption}
We note that since $\del_n(0)=0$ and $\del_n(z)$ is uniformly bounded it follows that $|\del_n(z)|\leq C'|z|\del^n$, and so, by possibly considering a smaller neighborhood of $0$ we can always assume that $|\del_n(z)|\leq \ve \del^n$ for an arbitrary small $\ve$.

\begin{lemma}\label{Approx Lemma}
For all $m$, all the conditions specified in Assumption \ref{GrowAssum}  are met under Assumption \ref{GenAss}.
 Moreover, 
 \begin{equation}\label{Relat0}
 \sig_n^2=n\Pi''(0)+H''(0)+O(\del^n).
\end{equation}

\end{lemma}
\begin{proof}
By taking the logarithm of both sides of \eqref{Stat assmp} we see that
$$
\Gamma_n(z):=\ln\bbE[e^{zS_n}]=n\Pi(z)+H(z)+r_n(z)
$$
where $r_{n}(z)=O(|z|\del^n)$ and $H(z)=\ln U(z)$ are both analytic functions.
Now \eqref{Relat0}  follow by differentiating $\Gamma_n(z)$, using the Cauchy integral formula and plugging in $z=0$. 
Note also that the same argument yields that for $j\geq3$, for all $|t|$ small enough we have
 $$
 \Lambda_n^{(j)}(t)=\sig_n^{-j}\Gamma_n^{(j)}(it/\sig_n)=\sig_n^{-j}\left(n\Pi^{(j)}(it/\sig_n)+H^{(j)}(it/\sig_n)\right)+O(\del^n)
 $$
 and since $\sig_n^2\asymp \Pi''(0)n$ we see that
  Assumption \ref{GrowAssum} holds true for every $m$. 
 \end{proof}

\subsection{Applications to homogeneous elliptic Markov chains}
In this section we will show that partial sums generated by geometrically ergodic homogeneous Markov chain and a single function $f$ satisfy the conclusion of Theorem \ref{BE2}. For this purpose we show that the conditions of the previous section are met. Again, the main innovation is that we are able to obtain non-uniform Berry-Esseen estimates. In the setup of this section uniform estimates follow from Nagaev's method \cite{Nag2} (see also \cite{HH}).

Let $(\xi_j)$ be an homogeneous (not necessarily stationary) Markov chain taking values on some measurable space $\cX$. 
Let $R$ be the corresponding Markov operator, that is $R$  maps a bounded function $g$ on $\cX$ to the function $Rg$ on $\cX$ given by
$$
Rg(x):=\bbE[g(\xi_{1})|\xi_0=x].
$$ 
Let us assume that $R$ has a continuous action on a Banach space $(B,\|\cdot\|)$ of function on $\cX$ with a norm $\|\cdot\|$ satisfying $\|g\|\geq \sup|g|$, and that for every function $f\in B$ and a complex number $z$ we have $e^{zf}\in B$ and that the function $z\to e^{zf}$ is analytic. Moreover, let us assume that $Rg$ is quasi compact and that the constant function $h_0=\textbf{1}$ is the unique eigenfunction with eigenvalue of modulus $1$. Namley, there is a unique stationary measure $\mu$   for the chain and constants $\del\in(0,1)$ and $C>0$ so that for every function $g\in B$ we have
$$
\|R^n g-\mu(g)\|\leq C\|g\|\del^n.
$$ 

\begin{example}
Suppose that the classical Doeblin condition holds true:  there exists a probability measure $\nu$ on $\cX$ and a constant $c>0$ so that for every measurable set $\Gamma\subset\cX$ we have
 $$
 \bbP(\xi_{n_0}\in\Gamma|\xi_0=x)\geq c\nu(\Gamma).
 $$
 Then the Markov chain is geometrically egrodic, that is $R$ satisfies the above properties with the norm $\|g\|=\sup|g|$. 
\end{example}

Next,  let $f\in B$  and set 
$$
S_n=\sum_{j=1}^n(f(\xi_j)-\bbE[f(\xi_j)]).
$$
 For a complex parameter $z$, let 
$$
R_z(g)=R_z(e^{zf}g).
$$
Then $z\to R_z$ is analytic. Thus by applying a holomorphic perturbation theorem we get that in a neighborhood of the origin there are analytic functions $\la(z),\nu^{(z)},h^{(z)}$ of $z$, so that  $\la(z)\in\bbC\setminus\{0\}$, $h^{(z)}\in B$, $\nu^{(z)}\in B^*$ ($B^*$ is the dual space) and $\la(0)=1, h^{(0)}=\textbf{1}$ and $\nu^{(0)}=\mu$, where $\textbf{1}$ is the constant function taking the value $1$.
 Moreover, 
 $$
R_z^{n}=\la^n(z)\left(\nu^{(z)}\otimes h^{(z)}+O(\del^n)\right)
 $$ 
 where $\big(\nu\otimes h\big)(g)=\nu(g)h$ and $\del\in(0,1)$.
 Thus, with $\Pi(z)=\ln\la(z)$ and $U(z)=\bbE[h^{(z)}(X_1)]$ we have 
 $$
 \bbE[e^{zS_n}]=\bbE[R_z^n\textbf{1}(X_1)]=e^{n\Pi(z)}\left(U(z)+O(\del^n)\right).
 $$
 We thus get the following result.
\begin{theorem}
Assumption  \ref{GenAss} is in force in the above circumstances. Therefore, with
$A_n=O(1)$ and either  $B_n=\sig\sqrt n+O(1)$ or $B_n=\sig_n+O(1)$,
 the non-uniform Berry-Esseen theorem  of any power holds true (i.e. Theorem \ref{BE2} for all $m$).
\end{theorem}
\begin{remark}
We would like also to refer to the first example in Remark \ref{Add} for related types of Markov chains for which Assumption  \ref{GenAss} (and hence Theorem \ref{BE2}) holds true.
\end{remark}
\begin{remark}\label{DiffRem}
  In order to verify Assumption \ref{GrowAssum} with  a given $m$, it is enough to assume that the operators $t\to R_{it}$ are  $C^{m+1}$ in a (real) neighborhood of the origin. Then, under the same conditions on $R$  we get that 
  \begin{equation}\label{lpp}
    R_{it}^n=\la(t)^n\nu^{(it)}\otimes h^{(it)}+N^n(t)=\la(t)^n\left(\nu^{(it)}\otimes h^{(it)}+\tilde N^n(t)\right)
  \end{equation}
  where  $\tilde N(t)=N(t)/\la(t)$ and $N(t)$ satisfies that
$\|N^n(t)\|\leq C\del^n$ for some $\del\in(0,1)$ (all the expressions are $C^{m+1}$ in $t$). Now, since $
\la(0)=1$, if $|t|$ is small enough we have  $\|\tilde N^n(t)\|\leq \del_1^n$  for some $\del_1\in(0,1)$. 
Since all the first $m+1$ derivatives of $\tilde N(t)^n$ decay exponentially fast to $0$ as $n\to\infty$ we can verify Assumption \ref{GrowAssum} for the given $m$ by differentiating the logarithms of both sides of \eqref{lpp}. 

The above setting includes the case when the two sided Doeblin conditions holds and $f\in L^{m+1}(\mu)$: there exists $c>1$ such that for every state $x$ and all measureble subsets $\Gamma$ of the state space we have
$$
c^{-1}\mu(\Gamma)\leq \mathbb P(\xi_{1}\in\Gamma|\xi_{0}=x)\leq c\mu(\Gamma)
$$
Indeed, in this case we can take $\|g\|=\sup|g|$ and $R_{it}$ is clearly of class $C^{m+1}$.
   \end{remark}

\subsection{A functional analytic setup with applications to hyperbolic dynamical systems}\label{Sec5.2}
In this section we will present a classical mechanism to verify Assumption \ref{GenAss}. The results in this section are well known, but we decided to present them to make the paper more self-contained. In the setting of this section the uniform Berry-Esseen theorem follow from the arguments in \cite{HH}, and our main novelty is that we obtain non-uniform Berry-Esseen theorems.

 Let $(B,\|\cdot\|)$ be a Banach space of functions on some measurable space $\cX$ so that the constant functions are in $B$ and  $\|g\|\geq\|g\|_{L^1(\mu)}$ for some probability measure $\mu$ which, when viewed as  linear functional, belongs to the dual $B^*$ of $B$.
   We also assume that $\|fg\|\leq C\|f\|\|g\|$ for some $C>0$ and all $f,g\in B$. 
 In addition, we  suppose that there is an operator $\cL$ acting on $B$  so that $\cL^*\mu=\mu$ and a sequence of $\cX$-valued random variables $\{U_j:j\geq1\}$ so that  for all $f_1,...,f_n\in B$ we have
 \begin{equation}\label{ExAss}
  \bbE\left[\prod_{j}f_j(U_j)\right]=\mu\left(\cL_{f_n}\circ\cdots\circ \cL_{f_2}\circ \cL_{f_1}h_0\right)
 \end{equation}
where  $h_0\in B$ is positive and satisfies $\mu(h_0)=1$, $\cL h_0=h_0$  and
 $$
 \cL_{f_j}(g):=\cL(f_j g).
 $$
 We then take $f\in B$ so that $\mu(f)=0$ and set $S_n=\sum_{j=1}^n X_j$, for $X_j=f(U_j)$. For each complex number $z\in\bbC$ we also set 
 $$
 \cL_z(g)=\cL_{e^{zf}}(g)=\cL(e^{z f}g).
 $$
 Let us also assume that\footnote{It will follow from the assumptions of Proposition \ref{PrP} that the limit $\sig^2=\Pi''(0)$ always exists, and so the main additional assumption here concerns the positivity of $\sig^2$. In our applications, a characterization for the latter positivity are well known (and will be recalled in the next sections).} 
 $$
\sig^2=\lim_{n\to\infty}\frac1n\text{Var}(S_n)>0.
 $$
We have the following.
 \begin{proposition}\label{PrP}
In the above circumstances,
 Assumption \ref{GenAss} is in force for all $m$ if
 $z\to\cL_z$ is well defined analytic in $z$ in a complex neighborhood of the origin, the operator $\cL$ is quasi-compact, the spectral radius of $\cL$ is $1$, and up to a multiplicative constant, the function $h_0$ is the unique eigenfunction corresponding to an eigenvalue of modules one.  
 \end{proposition}
 
\begin{proof}
By applying an analytic perturbation theorem we get that there is a number $r>0$ so that for every $z\in\bbC$ with $|z|<r$ there are analytic functions of $z\to \la(z)\in\bbC\setminus\{0\}$, $z\in h^{(z)}\in B$ and $z\to\nu^{(z)}\in B^*$ so that 
$$
\left\|\la(z)^{-n}\cL_z^n-\nu^{(z)}\otimes h^{(z)}\right\|\leq C\del^n
$$
where $C>0$, $\del\in(0,1)$ and the operator $\nu\otimes h$ is given by $g\to\nu(g)h$. Moreover, $\la(0)=1$, $\nu^{(0)}=\mu$ and $h^{(0)}=h_0$.
Thus,  
$$
\bbE[e^{zS_n}]=\mu(\cL_z^nh_0)=e^{n\ln\la(z)}\left(U(z)+\del_n(z)\right)
$$
where $U(z)=\nu^{(z)}(h_0)\mu(h^{(z)})$ (which indeed takes the value $1$ at $z=0$) and $\del_n(z)=O(\del^n)$ for some $\del\in(0,1)$. 
\end{proof}

  
  \begin{remark}
  Like in Remark \ref{DiffRem}, 
  in order to verify Assumption \ref{GrowAssum} with  a given $m$, it is enough to assume that the operators $t\to \cL_{it}$ are  $C^{m+1}$ in a (real) neighborhood of the origin. 
However, in the applications we have in mind the operators $z\to\cL_z$ will already be analytic.
   \end{remark}

\subsubsection{Application to expanding or hyperbolic dynamical systems}\label{SecDS}
In this section we apply the results of the previous section with several classes of chaotic dynamical systems. Again, the main innovation is that the non-uniform Berry-Esseen estimates as in Theorem \ref{BE2} hold. 

Let $(X,\cB,\mu)$ be a probability space and let $f:\cX\to\bbR$.  Let $T:\cX\to\cX$ be a measurable map  and set  
$$
U_j=T^j X_0
$$
where $X_0$ is a random element in $\cX$ whose distribution is $\mu$. Let $B$ be a Banach space of measurable functions on $\cX$ with the properties described in Section \ref{Sec5.2}.
Then the conditions of Proposition \ref{PrP} hold true for a variety of (non-uniformly) expanding maps and measures, some of which listed below:
\vskip0.1cm
\textbf{Subshifts of finite type and Anosov diffeomorphisms. }Recall that $(X,\cB,T)$ is a topologically mixing (one sided) subshift of finite type if there is a finite set $\mathcal A$ with cardinality larger than $2$, a matrix $A=(A_{i,j})_{i,j\in\mathcal A}$ with $0-1$  such that $A^M$ has only positive entries for some $M\in\mathbb N$,
$$
X=\{(x_j)_{j\geq0}: x_j\in\mathcal A, A_{x_{j},x_{j+1}}=1,\,\,\forall j\},
$$
$\cB$ is the restriction of the product topology and $T:X\to X$ is the left shift. Recall that (see \cite{Bowen}) for every $\alpha\in(0,1]$ and a H\"older continuous function $\varphi:X\to\mathbb R$ there exists a unique probability measure $\mu_{\varphi}$ on $X$ such that for every $x=(x_j)\in X$ and all $n\in\mathbb N$ we have 
$$
C^{-1}\lambda^ne^{\sum_{j=0}^{n-1}\varphi\circ T^j(x)}\leq \mu_\varphi([x_0,x_1,...,x_{n-1}])\leq C\lambda^ne^{\sum_{j=0}^{n-1}\varphi\circ T^j(x)}
$$
for some constants $C>1$ and $\lambda>0$ which depend only on $\varphi$. Here $[x_0,...,x_{n-1}]=\{(y_j)\in X: y_i=x_i, i\leq n-1\}$.  The measure $\mu=\mu_\varphi$ is called the Gibbs measure associated with $\varphi$.
In this setting we take $B$ to be the space of H\"older continuous functions, equipped with the norm 
$$
\|g\|=\sup|g|+v_\al(g)
$$
where $v_\al(g)$ is the H\"older constant of $g$ corresponding to the exponent $\al$. Then (see \cite{Bowen}) there is a strictly positive function $h$ with $\|h\|_\alpha<\infty$ such that the conditions in the previous section hold with $h_0=1$ and the operator $\mathcal L$ which maps a function $g:X\to\mathbb R$ to a function $\mathcal Lg$ given by
$$
\mathcal Lg(x)=\lambda^{-1}(h(x))^{-1}\sum_{y: Ty=x}e^{\varphi(y)}h(y)g(y). 
$$
We thus derive non-uniform Berry-Esseen estimate for partial sums of the form 
$$
S_n=\sum_{j=0}^{n-1}f\circ T^j
$$
under the assumption that  $\|f\|_\alpha<\infty$.

Next, we define two sided subshifts  of finite type $(\tilde X,\tilde\cB,\tilde T)$ similarly but with two sided sequences $(x_j)_{j\in\mathbb Z}$. Then by Sinai's lemma (see \cite{Bowen}) every H\"older continuous function on the two sided shift space differs from the natural lift of a  H\"older continuous function on the one sided shift by a term of the form $u-u\circ \tilde T$ with $u$ being H\"older continuous. Thus, Theorem \ref{BE2} for partial sums of the form 
$$
\tilde S_n=\sum_{j=0}^{n-1}\tilde f\circ \tilde T^j
$$
follow, under the assumption that  $\|\tilde f\|_\alpha<\infty$.
 This has applications to Anosov maps via symbolic representations, and we refer to \cite{Bowen} for the details. 
 
 Finally, note that in the above settings the optimal uniform Berry-Esseen theorem was proven in \cite{GH}, and the main novelty here is that we prove non-uniform Berry-Esseen estimates as in Theorem \ref{BE2}.
\vskip0.1cm
\textbf{Expanding maps on the unit interval and observables with bounded variation. }
Let $X=[0,1]$ and $\mathcal B=\text{Borel}$. Let $T:X\to X$ be a map for which there is a finite partition $\mathcal I$ (up to end points) of $[0,1]$ to closed intervals such that for each $I$ the map $T|I$ is of class $C^2$ and $\inf_{x\in I}|T'(x)|>1$. Let $v(f)$ be the usual variation of a function $f$ on $[0,1]$ and let $\|f\|=\sup|f|+v(f)$. Suppose next that for every interval $J\subset I$ there exists $M_I\in\mathbb N$ such that 
$$
T^{M_I}(I)=[0,1].
$$
Then (see \cite{KelLiv}) there exists a unique absolutely continuous invariant measure $\mu=h \,d\text{Leb}$, where $h$ satisfies $\|h\|<\infty$ and it is strictly positive.
Moreover, the assumptions of Proposition \ref{PrP} are met with the function $h_0=1$ and the operator
$$
\mathcal L g(x)=(h(x))^{-1}\sum_{y: Ty=x}\frac{1}{T'(y)}h(y)g(y).
$$
\vskip0.1cm

\textbf{Young towers and partially hyperbolic maps}
Let $(\Del_0,\cF_0,\nu_0)$ be a probability space, $\{\Del_0^j:\,j\geq1\}$ be a partition of $\Del_0$ (\text{mod} $\nu_0$), and $R:\Del_0\to\bbN$ be a (return time) function which is constant on each one of the $\Del_0^j$'s. Henceforth we asssume that $\gcd\{R(x): x\in\Delta_0\}=1$.
Next, we  identify each element $x$ in $\Del_0$ with the pair $(x,0)$, and
for each nonnegative integer $k$ let the $k$-th floor of the tower be defined by
\[
\Del_k=\{(x,k)\in\Del_0\times\{k\}:\,\,R(x)>k\}.
\]
For each $j$ so that $R|\Del_0^j>k$ set 
\[
\Del_k^j=\{(x,k)\in\Del_k:\,x\in\Del_0^j\}\subset\Del_k.
\]
The whole tower is defined by 
\[
X=\{(x,k):\,k\geq0,\,\,(x,k)\in\Del_k\}=\bigcup_{k\geq0}\Del_k.
\]
Let $f_0:\Del_0\to\Del_0$ be so that for each $j$ the map $f_0|\Del_0^j:\Del_0^j\to\Del_0$ is bijective (\text{mod} ${\nu_0}$). The dynamics on the tower is given by the map $T:X\to X$ defined by
\[
T(x,k)=\begin{cases}(x,k+1) & \text{ if }R(x)>k+1 \\(f_0(x),0)& \text{ if }R(x)=k+1\end{cases}.
\]
We think of $(f_0(x),0)$ as the return (to the base $\Del_0$) function corresponding to $F$, and when $R(x)=k+1$ we will also write $F^R(x,0):=F(x,k)=(f_0(x),0)$. It will also be convenient to set 
$F^R(x,k)=F^R(x,0)$ for any $k\geq1$ and $(x,k)\in\Del_k$.
We note that in applications usually $\Del_0$ is a subset of a larger set, and $f_0=f^R$ is the return time function (to $\Del_0$) of a different function $f$ (so that the tower is constructed in order to study statistical properties of $f$). We assume here that the partition $\cC=\{\Del_k^j\}$ is generating in the sense that 
\[
\bigvee_{i=0}^\infty T^{-i}\cC
\]
is a partition into points. For each $k\geq1$ and $x\in X$, we will denote the element of the partition 
\[
\cC_k=\bigvee_{i=0}^{k-1} T^{-i}\cC
\]
containing $x$ by $\cC_k(x)$ (so that $\{x\}=\cap_{k\geq0}\cC_k(x)$). The partition elements of $\cC_k$ are called cylinders of length $k$.

Next, we lift the $\sig$-algebra $\cF_0$ to $X$ by identifying $\Del_k^j$ with $\Del_0^j$ and lift the  probability measure $\nu_0$ to a measure on $X$, by assigning the mass $\nu_0(\Gam)$ to each subset $\Gamma$ of each $\Del_k^j$, for any $k$ and $j$ so that $R|\Del_0^j>k$. Let us denote the above $\sig$-algebra and measure on $X$  by $\cF$ and $m$, respectively.


The uniform distance $d_U$ on the space $X$ is defined as follows: for 
every $x$ and $y$ in $X$, we denote by $s_U(x,y)$ the greatest positive integer $n$ so that $T^px$ and $T^p y$ belong the same element of the partition $\{\Del_k^j\}$, for all $p<n$ (namely, they  belong to the same partition element in $\cC_{n}$  but not to the same element of $\cC_{n+1}$).  When $x=y$ we set $\beta(x,y)=\infty$.
Let $\be\in(0,1)$ and define a metric $d_U(\cdot,\cdot)$ on $X\times X$ by $d_U(x,y)=\beta^{s_U(x,y)}$ (where $\beta^\infty:=0$). We assume that for every $k,j$,
\[
T^R:\Del_k^j\to\Del_0
\]
and its inverse are both non-singular with respect to $m$, and   the (inverse) Jacobian $JT^R$ is logarithmically locally Lipschitz continuous in the sense that for all $k$ and $x,y\in\Del_k^j$,
\begin{equation}\label{Jack Reg1}
\left|\frac{JT^R(x)}{JT^R(y)}-1\right|\leq Cd_U(T^Rx,T^Ry)
\end{equation}
for some constant $C$ which does not depend on $k,j,x$ and $y$.
We remark that uniform towers arise as extensions for certain classes of partially hyperbolic diffeomorphisms after collapsing along stable manifolds, see \cite{Y1}. 

Then (see \cite{Y1}), in the above setting there is a unique   absolutely continuous (with respect to $m$) $F$-invariant probability measure $\mu=h\,dm$ with $h$ being strictly positive
and a Banach space $(B,\|\cdot\|)$ of  weighted H\"older continuous functions on $X$ such that  
$\|h\|<\infty$ and
the conditions  of Proposition \ref{PrP} hold with $h_0=1$ and the operator
\begin{equation}\label{P def}
\mathcal Lg(x)=(h(x))^{-1}\sum_{y\in T^{-1}\{x\}}\frac{g(y)h(y)}{J T(y)}
\end{equation}
where $JF=\frac{dT_*m}{dm}$.

This means that we get limit theorems for functions $f$ belonging to the space $B$. Next, invertible towers are defined similarly and Theorem \ref{BE2} for them is obtained by a version of Sinai's lemma (see \cite{Y1}) in that setup. Finally, note that Young towers have a variety of applications to partially hyperbolic or expanding maps, and we refer to \cite{Y1} for the details (see also \cite{MelNic} for a related setup).

\subsection{Application to products of random matrices}\label{SecMat}
Let $V$ be a $d$ dimensional vector space for some $d>1$. Let us fix some scalar product on $V$ and denote by $\|\cdot\|$ the corresponding norm. Let us denote by $X=\bbP V$ the projective space on $V$, equipped with a suitable Remannian distance $d(\cdot,\cdot)$ (see \cite[Chapter II]{BoLac}). Given $x\in V$ and a sequence $(g_n)$ of iid random variables which takes values on $G=GL(V)$, we define 
$$
S_n(x)=\log\frac{\|g_n\cdots g_1 x\|}{\|x\|}.
$$
Then, as usual, $S_n(x)$ can be represented as the ergodic sum of $\bar\phi(g,x)=\frac{\log \|gx\|}{\|x\|}$ (see \cite{HH}). 

Let $\mu$ be the common distribution of all $g_n$. Let us consider the following assumptions.
\begin{assumption}\label{Ass1 mat}[Exponential moments]
For some $\del>0$ we have 
$$
\int_G\max\left(\|g\|,\|g^{-1}\|\right)^\del d\mu(g)<\infty.
$$
\end{assumption}
and

\begin{assumption}\label{Ass2 mat}[Strongly irreducibility and proximal elements]
The semi group generated by the support $\Gamma_\mu$  of $\mu$ has the property that there is no finite union of proper subspaces which is $\Gamma_\mu$-invariant. Moreover, there exits $g\in \Gamma_\mu$ which has a simple dominant eigenvalue. 
\end{assumption}

As will be explained in what follows, using the arguments of  \cite{Guiv}, under the above two assumptions the following limits exist
$$
\la_1=\lim_{n\to\infty}\frac1n\bbE[\log\|g_n\cdots g_1\|]
$$
and 
$$
\sig^2=\lim_{n\to\infty}\frac1n\bbE[\left(\log\|g_n\cdots g_1\|-n\la_1\right)^2].
$$
Recall also that for every $x\in\bbP V$ we have 
\begin{equation}\label{ee}
\la_1=\lim_{n\to\infty}\frac1n S_n(x), \text{a.s.}
\end{equation}
and that the above limit is also in $L^1$, uniformly in $x$. In particular,
$$
\la_1=\lim_{n\to\infty}\frac1n\bbE[S_n(x)].
$$
Moreover, for every $x\in\bbP V$ we have
\begin{equation}\label{vv}
\sig^2=\lim_{n\to\infty}\frac1n\bbE[\left(S_n(x)-n\la_1\right)^2].
\end{equation}
The latter  actually also follows from the CLT for $n^{-1/2}(S_n(x)-\la_1 n)$ (see \cite{FP}) together with the existence of the limit $\lim_{n\to\infty}\frac1n\bbE[\left(S_n(x)-n\la_1\right)^2]$ which will be proven in the proceeding arguments. 
Finally, note also that, for more general potentials $\sig^2=0$ if and only if the potential $\bar\phi(g,x)$ admits an appropriate coboundary representation. For the above choice of $\bar\phi$ we have that $\sig^2>0$ (see \cite{HH, Guiv}).

\begin{proposition}\label{MatProp}
Under Assumptions \ref{Ass1 mat} and \ref{Ass2 mat} Assumption \ref{GrowAssum} is valid for all $m$.  Moreover,  we have $B_n:=\sig \sqrt n=\sig_n+O(1)$ and $A_n:=\la_1 n-\bbE[S_n]=O(1)$. Thus the sequence
$\frac{S_n(x)-n\la_1}{\sig\sqrt n}$  obeys the non-uniform Berry-Esseen theorem with all powers $m$, as in Theorem \ref{BE2}.
\end{proposition}

\begin{proof}
Let $\|\cdot\|_\ve$, $\ve\in(0,1]$ denote the usual H\"older norm  of functions on $X$ corresponding the the exponent $\ve$. 
For every sufficiently small complex $z$ let us define the linear operator
$$
\cL_zh(x)=\int_G e^{z(\bar\phi(g,x)-\la_1)}h(g\cdot x)d\mu(g).
$$
where $h$ is a bounded function. 
Since $\|g^{-1}\|\|x\|\leq \|gx\|\leq \|g\|\|x\|$ Assumption \ref{Ass1 mat} implies that there is a constant $\del_0>0$ so that if $|z|\leq \del_0$ then $\cL_z$ is well defined and acts continuously on the space of H\"older continuous functions (with respect to an arbitrary exponent $\ve$). Moreover, the map $z\to\cL_z$ is analytic. Now (see \cite{HH, Guiv}), under Assumptions \ref{Ass1 mat} and \ref{Ass2 mat} the operator $\cL_0$ is quasi-compact, its spectral radius is $1$, and the constant function $\textbf{1}$ is the unique eigenfunction with eigenvalue of modulus $1$ (note that $\cL_0\textbf{1}=\textbf{1}$). Now, using an analytic perturbation theorem (which follows from an analytic inverse function theorem for Banach space) we see that 
$$\cL_z^n=(\la(z))^n(\Pi_z+r_{z,n})$$ 
where $z\to\la_z$ is an analytic complex valued function, and $z\to\Pi_z$ and $z\to r_{z,n}$ are analytic operator-valued functions 
so that $\la(0)=1$, $\Pi_0(g)=\nu(g)\textbf{1}$ for some probability measure $\nu$, $|r_{z,n}|\leq Cr^n$ for some $r\in(0,1)$ and $C>0$. Moreover, $\Pi_z$ has the form $\Pi_z(g)=\nu^{(z)}(g)h^{(z)}$ for some analytic in $z$ element of the dual space of the space of H\"older continuous functions, and $h^{(z)}$ is a complex valued H\"older continuous function, with the map $z\to h^{(z)}$ being analytic and $h^{(0)}=\textbf{1}$.
 
Next, note that 
$$
\bbE[\exp(z(S_n(x)-n\la_1))\xi(g_n\cdots g_1x)]=\cL_z^n\xi(x).
$$
for every function $\xi$. Taking $\xi=\textbf{1}$ we get that, in a complex neighborhood of the origin we have 
\begin{equation}\label{BothSides}
\bbE[\exp(z(S_n(x)-\la_1 n))]=\cL_z^n\textbf{1}(x)=\la_z^n(U(z)+O(r^n))
\end{equation}
where $U(z)=\Pi_z(\textbf{1})=\nu^{(z)}(\textbf{1})h^{(z)}(x)$, $U(0)=1$. Since $\la(0)=1$ we can write
 $\la(z)=e^{\Pi(z)}$, where $\Pi(z)$ is an analytic function so that $\Pi(0)=0$. 
 
Next, let us show that Assumption \ref{GrowAssum} holds true for all $m$.
First, by taking the logarithms of both sides of \eqref{BothSides} and differentiating $k$ times for $k\geq1$ (using analyticity and the Cauchy integral formula) we see that 
$$
\bbE[S_n(x)]=n(\Pi'(0)+\la_1)+H'(0)+O(r^n),
$$
where $H(z)=\ln U(z)$, and
$$
\text{Var}(S_n(x))=\Pi''(0)n+H''(0)+O(r^n).
$$
By dividing both sides of the first equality by $n$, taking $n\to\infty$ and using \eqref{ee} we see that
$$
\Pi'(0)=0.
$$ 
and so
\begin{equation}\label{Exppp}
\bbE[S_n(x)]=n\la_1+O(1).
\end{equation}
Next, by dividing both sides of the second quality by $n$, taking $n\to\infty$ and using \eqref{vv} and \eqref{Exppp} we see that
$$
\Pi''(0)=\sig^2
$$
and so
$$
\text{Var}(S_n(x))=\sig^2n+O(1).
$$
Set $S_n=S_n(x)-\bbE[S_n]$. The above estimates show that the choices $B_n=\sig\sqrt n=\text{Var}(S_n)+O(1)$, and $A_n=n\la_1-\bbE[S_n]=O(1)$ fit our general framework described in Section \ref{Main}.

Finally,  by taking the logarithms of both sides of \eqref{BothSides} and differentiating both sides $j$ times for $j\geq3$ and taking into account that $\sig_n^2=\text{Var}(S_n(x))\asymp \sig^2 n$ 
we see that Assumption \ref{GrowAssum} is in force with all $m$.

\end{proof}



\section{Examples: Nonstationary dynamical systems and Markov chains}\label{NonStat}
In this section we will provide a few examples for which Assumption \ref{GrowAssum} is known to hold (yet, we provide most of the details for readers' convenience). In all these example the uniform Berry-Esseen estimates are known, and our main novelty is that we obtain the non-uniform estimates like in Theorem \ref{BE2}. 

\subsection{Inhomogenous uniformly elliptic Markov chains}\label{MC IH}
Let $(\cX_i,\cF_i),\,i\geq1$ be a sequence of measurable spaces.
For each $i$,  let $R_i(x,dy),\,x\in\cX_i$ be a measurable family of (transition) probability measures on $\cX_{i+1}$. Let $\mu_1$ be any probability measure on $\cX_1$, and let $X_1$ be an $\cX_1$-valued random variable with distribution $\mu_1$. Let $\{X_j\}$ be the Markov 
 starting from
 $X_1$ with the transition probabilities
\[
\bbP(X_{j+1}\in A|X_{j}=x)=R_{j}(x,A),
\] 
where $x\in\cX_j$ and $A\subset\cX_{j+1}$ is a measurable set.
Each $R_j$ also gives rise to a transition operator given by 
\[
R_j g(x)=\bbE[g(X_{j+1})|X_j=x]=\int g(y)R_j(x,dy)
\] 
which maps an integrable function $g$ on $\cX_{j+1}$ to an integrbale function on $\cX_j$ (the integrability is with respect to the laws of $X_{j+1}$ and $X_j$, respectively). We assume here that 
there are probability measures $\fm_j$,
$j>1$ on $\cX_j$ and families of transition probabilities $p_j(x,y)$ 
so that 
\[
R_j g(x)=\int g(y)p_j(x,y)d\fm_{j+1}(y).
\]
Moreover, there exists $\ve_0>0$ so that for any $j$ we have 
\begin{equation}
\label{DUpper}
 \sup_{x,y}p_j(x,y)\leq 1/\ve_0,
 \end{equation}
  and the transition probabilities of the second step
   transition operators $R_j\circ R_{j+1}$ of $X_{j+2}$ given $X_j$ are bounded from below by $\ve_0$ (this is the uniform ellipticity condition): 
\begin{equation}
\label{DLower}
\inf_{j\geq1}\inf_{x,z}\int p_j(x,y)p_{j+1}(y,z)d\fm_{j+1}(y)\geq \ve_0.
\end{equation}
Next, for a uniformly bounded sequence of measurable functions $f_j:\cX_j\times\cX_{j+1}\to\bbR$ 
we set $Y_j=f_j(X_j,X_{j+1})$ and
\begin{equation}
\label{AddFunct}
S_n=\sum_{j=1}^n(Y_j-\bbE(Y_j)).
\end{equation}
Set $V_n=\text{Var}(S_n)$ and $\sig_n=\sqrt{V_n}$. Then by \cite[Theorem 2.2]{DS} we have 
$\DS \lim_{n\to\infty}V_n=\infty$ if and only if one can not decompose $Y_j$ as
$$ Y_j=\bbE(Y_j)+a_{j+1}(X_{j+1},X_{j+2})-a_{j}(X_j, X_{j+1})+g_n(X_j, X_{j+1})$$ 
 where  $a_n$ are uniformly bounded functions and $\DS \sum_j g_j(X_j, X_{j+1})$ converges almost surely.
 Note that when the chain is one step elliptic (i.e. $\ve_0\leq p_j(x,y)\leq \ve_0^{-1}$) and $f_j(x,y)=f_j(x)$ depends only on the first variable then by \cite[Proposition 13]{PelPTRF} we have 
$$
C_1\sum_{j=1}^{n}\text{Var}(f_j(X_j))\leq \sig_n^2\leq C_2\sum_{j=1}^{n}\text{Var}(f_j(X_j))
$$  
for some constants $C_1,C_2>0$, and so $\sig_n\to\infty$ iff the series $\sum_{j=1}^{\infty}\text{Var}(f_j(X_j))$ diverges.


Next, in \cite[Proposition 24]{DolgHaf} we have shown that Assumption \ref{GrowAssum} holds true for all $m$. Thus we get the following result.

\begin{theorem}
The Berry-Esseen type theorem \ref{BE2} and holds true for every $m$. Moreover, the expectation estimates in Corollary \ref{Cor} hold true. 
\end{theorem}

\begin{remark}
As noted in Remark \ref{Add}, when $\sig_n^2$ grows linearly fast we get Theorem \ref{BE2} for more general chains ($\phi$-mixing). However, in this section we are interested in the situation when $\sig_n^2$ can grow arbitrary slow.
\end{remark}

\subsection{Random (partially expanding or hyperbolic) dynamical systems}\label{RDS}

Let $(\Om,\cF,\bbP)$ be a probability space and let $\theta:\Om\to\Om$ be an ergodic invertible map.
Let $(\cX,d)$ be a metric space and let $\cE_\om, \om\in\Om$ be a measureable family of compact subsets (see \cite[Ch. 6]{HK}). Set $\cE=\{(\om,x): \om\in\Om, x\in\cE_\om\}$ and let $T:\cE\to\cE$ be a measurable map of the form
$$
T(\om,x)=(\theta\om, T_\om x)
$$
where
$T_{\om}:\cE_\om\to\cE_{\theta\om}$ is a random family of maps ($T$ is the so-called skew product induced by $\{T_\om\}$). Let $\mu$ be a $T$-invariant probability measure on $\cE$, and let us represent it in the form 
$$
\mu=\int \mu_\om d\bbP(\om)
$$ 
where for $\bbP$-almost every $\om$ the probability measure $\mu_\om$ on $\cE_\om$ satisfies $(T_\om)_*\mu_\om=\mu_{\te\om}$. 

Next, let $\cL_\om$ denote the dual of $T_\om$ with respect to the measures $\mu_\om$ and $\mu_{\theta\om}$, namely the unique operator so that 
$$
\int g\cdot(f\circ T_\om)d\mu_\om=\int f\cdot(\cL_\om g)d\mu_{\te\om}
$$
for all bounded functions $g$ and $f$ on the appropriate domains.
Let $(B_\om,\|\cdot\|_\om)$ be a norm on functions on $\cE_\om$ and let $f:\cE\to\bbR$ be a measurable function so that $\|f\|:=\text{ess-sup}\|f(\om,\cdot)\|_\om<\infty$. In what follows we will introduce assumptions on $\cL_\om$ and their complex perturbations which will guarantee that Assumption \ref{GrowAssum} holds true for 
$$
S_n^\om=\sum_{j=0}^{n-1}f_{\te^j\om}\circ T_{\te^{j-1}\om}\circ\cdots\circ T_{\om},\,\,\,f_\om(\cdot)=f(\om,\cdot)
$$
for $\bbP$-a.e. $\om$,  when considered as a random variables on the space $(\cE_\om,\mu_\om)$.

For a complex parameter $z$, let us consider the operator $\cL_\om^{z}$ given by $\cL_\om^{z}(g)=\cL_{\om}(ge^{zf_\om})$. For it to be well defined, we  assume that $\cL_\om$ is a continuous operator between $B_\om$ and $B_{\te\om}$ and that the map $z\to e^{zf_\om}\in\cB_\om$ is analytic in $z$, uniformly in 
$\om$.

\begin{assumption}\label{AssRDF}
There is a constant $r_0>0$ so that ($\bbP$-a.s.) for every complex $z$ with $|z|<r_0$ there is a triplet $\la_\om(z)\in\bbC\setminus\{0\}$, $h_\om^{(z)}\in B_\om$ and $\nu_\om^{(z)}\in B_\om^*$ which is measurable in $\om$, analytic in $z$ and: 
\begin{enumerate}
\item $\la_\om(0)=1$, $h_\om^{(0)}=\textbf{1}$, $\nu_\om^{(0)}=\mu_\om$, $\nu_\om^{(z)}(h_\om^{(z)})=1$;
\vskip0.1cm
\item with 
$$\cL_{\om}^{z,n}=\cL_{\te^{n-1}\om}^{z}\circ\cdots\circ\cL_{\om}^{z}$$
and 
$$
\la_{\om,n}(z)=\la_{\te^{n-1}\om}(z)\cdots\la_{\om}(z)
$$
there are $C>0$ and $\del\in(0,1)$ so that
$$
\left\|(\la_{\om,n}(z))^{-1}\cL_{\om}^{z,n}-\nu_{\om}^{(z)}\otimes h_{\te^n\om}^{(z)}\right\|_{\te^n\om}\leq C\del^n
$$
where the operator $\nu\otimes h$ is given by $g\to \nu(g)\cdot h$.
\end{enumerate}
\end{assumption}

\begin{example}
Let us list a few types of random maps $T_\om$ and measures $\mu_\om$ for which Assumption \ref{AssRDF} holds true.
\begin{enumerate}
\item The random expanding maps  considered in \cite[Ch. 6]{HK}, where $\|\cdot\|_\om$ is a H\"older norm of a function on $\cE_\om$ with respect to some exponent $\al\in(0,1]$ and $\mu_\om$ is a random Gibbs measure;
\vskip0.1cm
\item The random expanding maps  considered in \cite{DavorCMP}, where $\|\cdot\|_\om$ is the bounded-variation norm and $\mu_\om$ is the unique random  absolutely continuous equivariant measure.
\vskip0.1cm
\item The hyperbolic maps considered in \cite{DavorTAMS} (see also \cite{DDH}), with $\|\cdot\|_\om$ being the appropriate strong norm and $\mu_\om$ being the unique random fiberwise absolutely continuous equivariant measure.
\vskip0.1cm
\item The random partially expanding or hyperbolic maps considered in \cite{HafYT}, with $\|\cdot\|_\om$ being the appropriate ``weighted" H\"older norm and $\mu_\om$ being a sampling measure. 
\vskip0.1cm
\item The uniformly random version (as in \cite{RPF2022}) of the random partially expanding maps considered in \cite{Varandas}  with $\|\cdot\|_\om$ being a H\"older norm with respect to some exponent $\al\in(0,1]$ and $\mu_\om$ being a random Gibbs measure corresponding to a potential with a sufficiently small oscillation (see \cite{RPF2022}). 
\end{enumerate}
\end{example}

\begin{remark}
We note that in \cite{DavorCMP, DavorTAMS} Assumption \ref{AssRDF} appears as 
$$
\left\|\cL_{\om}^{z,n}-\la_{\om,n}(z)\left(\nu_{\om}^{(z)}\otimes h_{\te^n\om}^{(z)}\right)\right\|_{\te^n\om}\leq C\del^n
$$
but since $\la_\om(z)=1+O(z)$ upon decreasing $r_0$ we get that  $(1-\ve)^n\leq |\la_{\om,n}(z)|\leq (1+\ve)^n$ for an arbitrary small $\ve$, and hence the above two forms are equivalent. 
\end{remark}

\begin{proposition}\label{FirstProp}
Under assumption \ref{AssRDF} we we have tha following.

(i) There is a nonnegative number $\sig$ so that $\bbP$-a.s. 
$$
\sig^2=\lim_{n\to\infty}n^{-1}\text{Var}_{\mu_\om}(S_n^\om).
$$
 Moreover, $\sig^2=0$ iff $f(\om,x)-\int f(\om,y)d\mu_\om(y)=r(\om,x)-r(\te\om, T_\om x)$ for some function $r$ so that $\int|r(\om,y)|^2d\mu_\om(y)d\bbP(\om)<\infty$.

(ii) Suppose that $\sig^2>0$. Then, under  Assumption \ref{AssRDF}, for $\bbP$-a.e. $\om$ we have that
 $S_n=S_n^\om-\mu_\om(S_n^\om)$ verifies Assumption \ref{GrowAssum} with every $m$.  Thus, all the results stated in Theorem \ref{BE2} and Corollary \ref{Cor}  hold true for all $m$.
\end{proposition}
\begin{proof}
We have 
$$
\bbE[e^{zS_n^\om}]=\mu_{\te^n\om}(\cL_\om^{z,n}\textbf{1}).
$$
Now, like in the stationary case, we have 
\begin{equation}\label{Ab}
\ln \bbE[e^{zS_n^\om}]=\Pi_{\om,n}(z)+H_{\om,n}(z) +O(|z|\del^n).
\end{equation}
where\footnote{$\ln\la_\om(z)$ is an analytic branch which vanishes at the origin and is uniformly bounded around the origin.} $\Pi_{\om,n}(z)=\sum_{j=0}^{n-1}\ln\la_{\te^j\om}(z)$ and $H_{\om,n}(z)$ is bounded by some constant $C$ and it vanishes at $z=0$. Now, by differentiating both sides of \eqref{Ab},  using the Cauchy integral formula, and plugging in $z=0$ we see that 
$$
\sig_{\om,n}^2=\text{Var}_{\mu_\om}(S_n^\om)=\Pi_{\om,n}''(0)+O(1)=\sum_{j=0}^{n-1}\Pi_{\te^j\om}''(0)+O(1).
$$
Next, by ergodicity of $(\Om,\cF,\bbP,\te)$ we conclude that, $\bbP$-a.s. we have 
$$
\lim_{n\to\infty}\frac1n\sig_{\om,n}^2=\int_{\Om}\Pi_\om''(0)d\bbP(\om):=\sig^2.
$$
The proof of the characterization of positivity of $\sig^2$  proceeds similarly to \cite[Section 3]{Kifer1998}.

Finally, in order to show that Assumption \ref{GrowAssum} is satisfied for all $m$ we  differentiate both sides of \eqref{Ab}\, $j$ times and then use the Cauchy integral formula to conclude that for all $j\geq 3$ there are constants $r_j,C_j>0$ so that ($\bbP$-a.s.) on the domain $\{|z|\leq r_j\}$ we have
$$
\tilde\Lambda_n^{(j)}(z)\leq C_j n,\,\,\,\tilde\Lambda_n(z)=\tilde\Lambda_{\om,n}(z)=\ln \bbE[e^{zS_n^\om}]
$$
where we have used that $\ln\la_{\te^j\om}(z)$ and $H_{\om,n}(z)$ and all their derivatives are uniformly bounded in $z$, $\om,j$ and $n$ (when $|z|$ is small enough).
To complete the proof that Assumption \ref{GrowAssum} is satisfied we use that 
 $\sig_{\om,n}^2$ grows linearly fast in $n$ 
\end{proof}

\subsection{Sequential dynamical systems in a neighborhood of a single system}\label{SDS sec}
Let $(\cX_j,\cB_j,\mu_j)$ be a sequence of probability  space, and let $T_j:\cX_j\to\cX_{j+1}$ be a sequence of measurable maps. Let $\cL_j$ an operator which is defined by the duality relation:
$$
\int_{\cX_j}g\cdot (f\circ T_j)d\mu_{j}=\int (\cL_jg)\cdot fd\mu_{j+1}
$$    
for all bounded functions $g:\cX_{j}\to\bbR$ and $f:\cX_{j+1}\to\bbR$. Let $B_j$ be a Banach space of measurable functions on $\cX_j$, equipped with a norm
$\|\cdot\|_j$ so that $\sup|g|\leq C\|g\|_j$ for some constant $C$ which does not depend on $j$. We assume here that there are constants $A$ and $\del\in(0,1)$ and strictly positive functions $h_j\in B_j$ so that for every $g\in B_j$ and all $n\geq1$ we have
$$
\|\cL_{j+n-1}\circ\cdots\circ\cL_{j+1}\circ\cL_j g-\mu_j(g)h_{j+n}\|_{j+n}\leq A\|g\|_j\del^n.
$$
Let take now a sequence of real-valued functions $f_j\in B_j$ so that $\|f\|:=\sup_j\|f\|_j<\infty$, $\mu_j(f_j)=0$ and define 
$$
S_n=\sum_{j=0}^{n-1}f_j\circ T_{j-1}\circ\cdots\circ T_1\circ T_0(x_0)
$$
where $x_0$ is distributed according to $\mu_0$. Let us define $\cL_j^{(z)}(g)=\cL_j(e^{zf_j})$, where $z\in\bbC$. We further assume here that the map $z\to\cL_j^{(z)}$ is analytic around the complex origin (with values in the space of bounded linear operators between $B_j$ and $B_{j+1}$), uniformly in $j$.   
\begin{example}
\,
\begin{itemize}
\item[(i)] These conditions hold true in the setup of \cite{HaNonl}, where the maps $T_j$ are locally expanding, the space $B_j$ is the space of H\"older continuous functions and $\mu_j$ are measures so that $(T_j)_*\mu_j=\mu_{j+1}$. Note that when the maps are absolutely continuous with respect to some reference measure (e.g. a volume measure) then $\mu_j$ is equivalent to that measure. 
\vskip0.1cm
\item[(ii)] 
These conditions hold true in the setup of \cite{ConzeR}, where $\cX_j=\cX$ and $\mu_j=\mu$ coincide with the same manifold and (normalized) volume measure, respectively, each $T_j$ is a locally expanding map and $B_j=B$ is the space of functions with bounded variation.
\end{itemize}
\end{example}

Next, by applying  a sequential perturbation theorem based on \cite{Frank} or \cite[Theorem D.2]{DolgHaf PTRF 2}
 (or, in the setup of \cite{HaNonl} using complex cones contractions) we see that there is a constant $r_0$ so that for every complex parameter $z$ with $|z|\leq r_0$ there are  analytic in $z$ triplets $(\la_j(z),h_j^{(z)},\nu_j^{(z)})$ consisting of a complex non-zero random variable $\la_j(z)$, a complex-valued function $h_j^{(z)}\in B_j$ and a complex continuous linear functional $\nu_j^{(z)}\in B_j^*$ so that $\la_j(0)=1$, $h_j^{(0)}=h_j$ and $\nu_j^{(0)}=\mu_j$, $\nu_j^{(z)}(h_j^{(z)})=\nu_j^{(z)}(\textbf{1})=1$. Moreover, there are constants $A_1>0$ and $\del_1\in(0,1)$ so that for all $j,n$ and a function $g\in 	B_j$ we have
$$
\left\|\frac{\cL_{j+n-1}\circ\cdots\circ\cL_{j+1}\circ\cL_j g}{\la_{j+n-1}(z)\cdots \la_{j+1}(z)\la_j(z)}-\nu_j^{(z)}(g)h_{j+n}^{(z)}\right\|\leq A_1\|g\|\del_1^n.
$$ 
The following result follows exactly like the corresponding result about random dynamical systems in the previous section.

\begin{theorem}
Suppose that $\text{Var}_{\mu_0}(S_n)\geq c_0 n$ for some $c_0$ and all $n$ large enough. Then Assumption \ref{GrowAssum} is in force, and so Theorem \ref{BE2} and Corollary \ref{Cor} hold true.
\end{theorem}

\subsubsection{Linearly growing variances: the perturbative approach}
Let us suppose that $\cX_j=\cX$ and $B_j=B$  coincide with a single space $\cX$ and a Banach space, respectively. Let $T:\cX\to\cX$ be a map so that all the previous properties are satisfied with the sequence $T_j=T$, and let $f\in B$. Then the corresponding triplet $\la_{j,T}(z)=\la(z;T), h_{j,T}^{(z)}=h_T^{(z)}$ and $\nu_{j,T}^{(z)}=\nu^{(z)}_T$ does not depend on $j$. 
Next, let us assume that $f$ is not an $L^2(\mu_T)$ coboundary with respect to $T$, where $\mu_T=\nu_T^{(0)}$. Then, with $S_nf=\sum_{j=0}^{n-1}f\circ T^j$ we have 
$$
\sig^2=\lim_{n\to\infty}\frac1n\text{Var}_{\mu_T}(S_n f)>0.
$$
Let us denote the dual operator corresponding to $T$ by $\cL_T$.

\begin{proposition}\label{ppp}
Set 
$$
\ve_0=\sup_{j}\max\left(\|\cL_j-\cL_T\|, \|f-f_j\|\right).
$$
If $\ve_0$ is small enough then for all $n$ large enough we have $\sig_n^2\geq \frac12\sig^2$.
\end{proposition}

\begin{proof}
First, as was shown in \cite{HaNonl} (or using a sequential inverse function theorem in the more general setup), we have that 
$$
\sup_{j,|z|\leq r_1}\max\left(|\la_j(z)-\la_T(z)|, \|h_j^{(z)}-h_T^{(z)}\|, \|\nu_j^{(z)}-\nu_T^{(z)}\|\right)\leq\del(\ve_0)
$$
where $\del(\ve_0)\to 0$ as $\ve_0\to 0$, and $r_1$ is a constant. Next, as was shown in Section \ref{Sec5.2}, in the stationary case we have   $\sig^2=\Pi_{T}''(0)$ where $\Pi_T(z)$ is a branch of $\ln \la_T(z)$. Thus, arguing as in the proof of \cite[Theorem 2.4 (b)]{HaNonl} we see that 
$$
|\sig_n^2-\text{Var}_{\mu_T}(S_n f)|\leq n\del_1(\ve_0)
$$
 where $\del_1(\ve_0)\to0$ as $\ve_0\to 0$. Thus by taking $\ve_0$ small enough (so that $\del_1(\ve_0)\leq \sig^2/4$) we  conclude that for all $n$ large enough we have $\sig_n^2\geq c_0n$, where $c_0=\frac12\sig^2$.
\end{proof}

\begin{remark}
The transfer operators $T\to\cL_T$ are continuous with respect to the dynamics $T$ for a variety of piecewise monotone one dimensional interval maps when all $T$'s are close to a single map $T$ in the Keller-Liverani norm (see \cite{KelLiv1}). In particular, we can consider the case when $T_j x=\beta_j x$ and $Tx=\beta x$. Then the conditions of Proposition \ref{ppp} hold true when $\sup_j|\beta_j-\beta|$ and $\sup_j|f_j-f|$ are small enough. In the multidimensional case, the conditions of Proposition \ref{ppp}  hold true when each $T_j$ is obtained from $T$ by a small perturbation of each inverse branch of $T$. For instance, let $\cX=[0,1)^d$ for some $d\in\bbN$ and  suppose that there is a partition of $[0,1)^d$ into cubes of the form $[a_1,b_1)\times\cdots\times[a_d,b_d)$ so that on each cube $T$ is expanding (and say, has a full image). Now, the conditions of Proposition \ref{ppp} are in force  if each $T_j$ is generated by gluing sufficiently small perturbations of each one of the latter expanding $T|_{I}$ maps on the elements $I$ of the partition.
\end{remark}

\subsubsection{Products of random non-stationary matrices in a 
``neighborhood" of a single random matrix}
In this section we will briefly discuss how to extend the ideas from the previous section to uniformly bounded products of independent but not identically distributed random matrices, whose probability laws belong to a sufficiently small neighborhood of a single probability law with the properties described in Section \ref{SecMat}. 

Let us recall that the transfer operator corresponding to a probability distribution on $G=GL(V)$ is given by 
$$
\cL_\mu h(x)=\int_{G}h(g\cdot x)d\mu(g).
$$
Then, for every two probability laws $\mu_1,\mu_2$ on $G$ we have 
$$
\sup_x|\cL_{\mu_1}h(x)-\cL_{\mu_2}h(x)|\leq \sup|h|\|\mu_1-\mu_2\|_{TV}.
$$
Moreover, given two vectors $x,y\in V$, for every H\"older continuous function  $h$ with exponent $\al$ we have
$$
\left|\left(\cL_{\mu_1}h(x)-\cL_{\mu_1}h(y)\right)-\left(\cL_{\mu_2}h(x)-\cL_{\mu_2}h(y)\right)\right|=\left|\int \left(h(g\cdot x)-h(g\cdot y)\right)\left(d\mu_1(g)-d\mu_2(g)\right)\right|$$$$\leq\sup_{g\in S(\mu_1)\cup S(\mu_2)}|h(g\cdot x)-h(g\cdot y)|\|\mu_1-\mu_2\|_{TV}\leq \sup_{g\in S(\mu_1)\cup S(\mu_2)}\|g(x-y)\|^\al v_\al(h)|\|\mu_1-\mu_2\|_{TV}
$$
where $S(\mu)$ denotes the support of $\mu$. We thus see that 
$$
v_\al(\cL_{\mu_1}h-\cL_{\mu_1}h)\leq \|\mu_1-\mu_2\|_{TV}v_\al(h)\sup_{g\in S(\mu_1)\cup S(\mu_2)}\|g\|^\al.
$$
We conclude that 
\begin{equation}\label{Stability}
\|\cL_{\mu_1}-\cL_{\mu_2}\|_\al\leq\max\left(1,\sup_{g\in S(\mu_1)\cup S(\mu_2)}\|g\|^\al\right)\|\mu_1-\mu_2\|_{TV}. 
\end{equation}

Now, let us take an independent sequence $g_1,g_2,...$ of uniformly bounded random matrices which are not necessarily uniformly distributed. Let us denote that law of $g_i$ by $\mu_i$. 
Let $\mu$ be a  probability distribution  with the properties described at the beginning of Section \ref{SecMat} (so that Assumptions  \ref{Ass1 mat} and \ref{Ass2 mat} are in force and Proposition \ref{MatProp} holds true). Let us further assume that $S(\mu)$ is bounded.
Set
$$
\ve=\sup_i\|\mu_i-\mu\|_{TV}.
$$
 Then by \eqref{Stability},
$$
\lim_{\ve\to0}\sup_i\|\cL_{\mu_i}-\cL_{\mu}\|_\al=0.
$$
 Thus, as in the previous section, when $\ve$ is small enough we get an appropriate complex sequential spectral gap for the complex perturbations of the operators. This leads to the non-uniform Berry-Esseen theorem of all powers $m$ (and all their applications). That is, Theorem \ref{BE2} and  Corollary \ref{Cor} hold with all $m$.

\end{document}